\documentclass[11pt]{amsart}

\usepackage[a4paper,margin=1.1in]{geometry}
\usepackage[T1]{fontenc}
\usepackage{amssymb,amsmath,amscd,amsthm,mathtools}
\usepackage{tikz-cd}
\usepackage[dvipsnames]{xcolor}
\usepackage{hyperref}
 \hypersetup{colorlinks=true,linkcolor=Red,citecolor=Red,urlcolor=Red}
\usepackage[nameinlink]{cleveref}
\usepackage[utf8]{inputenc}  
\usepackage[english]{babel}
\usepackage{lmodern}
\usepackage{microtype}
\usetikzlibrary{arrows.meta,calc,positioning,decorations.pathreplacing}
\usetikzlibrary{calc}

\theoremstyle{plain}
\newtheorem{theorem}{Theorem}[section]

\newtheorem{proposition}[theorem]{Proposition}
\newtheorem{lemma}[theorem]{Lemma}
\newtheorem{corollary}[theorem]{Corollary}

\theoremstyle{definition}
\newtheorem{definition}[theorem]{Definition}
\newtheorem{remark}[theorem]{Remark}

\newtheorem{notation}[theorem]{Notation}

\DeclareMathOperator{\Dgn}{Dgn}      
\DeclareMathOperator{\Mat}{Mat}      
\DeclareMathOperator{\Id}{Id}       
 
\newcommand{\mon}{\mathrm{mon}}
\newcommand{\Mon}{\mathrm{Mon}}
\newcommand{\fol}{\mathcal{F}}

\newcommand{\Blor}{\mathrm{Bl}^{\mathrm{or}}}
\newcommand{\ZZ}{\mathbb Z}
\newcommand{\CC}{\mathbb C}
\newcommand{\GL}{\mathrm{GL}}

\theoremstyle{definition}

\newcommand{\C}{\mathbb{C}}

\newcommand{\OO}{\mathcal{O}}
\newcommand{\End}{\mathrm{End}}
\newcommand{\Hom}{\mathrm{Hom}}

\newcommand{\im}{\mathrm{Im}}

\newcommand{\Res}{\mathrm{Res}}

\title[A Poisson--Poincar\'e--Dulac for Poisson Connections ]
{  A Poisson--Poincar\'e--Dulac for Poisson Connections }

\author[Maur\'icio Corr\^ea]{ Maur\'icio Corr\^ea }
\address[Maur\'icio  Corr\^ea]{Dipartimento di Matematica, Universit\`a degli Studi di Bari,  Via E. Orabona 4, I-70125, Bari, Italy
}
\email{mauricio.correa.mat@gmail.com, mauricio.barros@uniba.it}

\author{Miguel Rodr\'iguez Pe\~na}
\address{ Miguel Rodr\'iguez Pe\~na \\
ICEX-UFMG, Departamento de Matem\'atica,
Av. Ant\^onio Carlos, 6627, 31270-901  Belo Horizonte, MG, Brazil.}
\email{amrp2024@ufmg.br}
\date{\today}

\begin{document}

\begin{abstract}
We study Poisson-flat connections with logarithmic poles along a simple normal crossings divisor on a holomorphic Poisson manifold, where flatness is required only along the symplectic foliation. After identifying the relevant logarithmic cotangent Poisson Lie algebroid, we define an Euler--Poisson principal part and a residue theory adapted to the canonical logarithmic Hamiltonian generators. Under a precise effective nonresonance hypothesis, we establish a Poisson Poincar\'e--Dulac theorem: any logarithmic Poisson-flat connection with fixed Euler--Poisson principal part, equivalently with residue-free positive boundary remainder, is holomorphically gauge equivalent to a pure Euler--Poisson normal form with constant commuting residues, and this normal form is unique up to Casimir-valued gauge transformations lying in the effective centralizer of the residues.
To encode both leafwise transport and boundary winding, we construct a twisted leafwise fundamental groupoid via the real oriented blow-up and a 2-pushout that adjoins canonical tangential meridians. In the effectively non-resonant regime, the normal form yields a meridional character determined by the residues, and records the meridional boundary part of the logarithmic Riemann--Hilbert picture in the Poisson setting. Finally, we illustrate the theory with rank-two Poisson modules (Poisson triples) and an explicit local affine family of examples.
\end{abstract}

\maketitle
\section{Introduction}

 Let $(X,\sigma)$ be a holomorphic Poisson manifold. On the Poisson regular locus $X^\circ$ the anchor
$\sigma^\#:\Omega^1_{X^\circ}\to T_{X^\circ}$ has constant rank, and its image
$\fol_\sigma:=\mathrm{Im}(\sigma^\#)$ defines an integrable holomorphic distribution whose leaves
carry holomorphic symplectic forms. Holomorphic Poisson geometry furnishes a natural framework for
integrability and deformation problems, and for algebro-geometric and representation-theoretic
structures such as Poisson modules; see, for instance, \cite{Polishchuk97,LGV13,Vai94}. From a complementary perspective, Poisson structures lie at the heart of deformation quantization
and its geometric realizations \cite{Kon03,CF00}, and they also arise prominently in mathematical
physics, for example in the appearance of noncommutative geometries induced by $B$-fields in string
theory \cite{SW99}; see also Hitchin \cite{Hit11} for the $B$-field action in the theory of
generalized holomorphic bundles.

In this work we study logarithmic Poisson-flat connections with poles along a simple normal
crossings divisor $D\subset X$. After a choice of local frame, such a connection may be represented
by a matrix $\Theta$ whose entries are logarithmic Hamiltonian vector fields tangent to $\fol_\sigma$,
and Poisson-flatness is expressed by the Maurer--Cartan (zero-curvature) equation in the Poisson
direction. The guiding local analytic question is the Poisson counterpart of the regular-singular
normalization problem: \textit{for a fixed logarithmic principal part, can one eliminate higher-order terms
by holomorphic gauge transformations, and to what extent is the resulting normal form unique?
}

For ordinary logarithmic flat connections this circle of questions is governed by the
Poincar\'e--Dulac philosophy: under suitable non-resonance assumptions one solves the
relevant homological equations order-by-order and proves convergence, leading to holomorphic gauge
normal forms together with a residue and monodromy calculus; see Deligne \cite{Deligne70} and the
gauge-oriented account of Novikov--Yakovenko \cite{NY02}. A closely related perspective regards logarithmic connections as representation-theoretic data for Lie groupoids integrating the logarithmic tangent Lie algebroid $T_X(-\log D)$; see \cite{Bischoff20} for the groupoid-theoretic framework and \cite{Bischoff22} for the corresponding normal form theory and the construction of moduli stacks. We also keep in mind the general background on Lie groupoids,
foliations, and their relation to Lie algebroids \cite{MM03}, as well as stack-theoretic foundations
relevant to groupoid moduli \cite{Noohi}. In the irregular setting, Stokes-type groupoids play an
analogous role; see \cite{GLP18} and the systematic exposition in \cite{Sabbah13}.

Because our local models are formulated near boundary strata of a simple normal crossings divisor,
we work naturally with the real oriented blow-up and its toroidal collar charts, as surveyed in
\cite{PopescuPampu25}. Related logarithmic topological constructions, and the associated log Betti
viewpoint, are developed in the Kato--Nakayama framework \cite{KN99}, which provides a useful
conceptual backdrop for the appearance of angular variables and boundary strata.

The Poisson-logarithmic setting is technically distinct from the ordinary one. The flatness
directions are not those of $T_X(-\log D)$, but rather those dictated by the symplectic foliation
$\fol_\sigma$, whose rank equals the Poisson rank and may vary near the degeneracy locus of $\sigma$.
Accordingly, even locally one cannot simply reduce to the usual logarithmic connection theory: the
homological equations must be solved along $\fol_\sigma$, and one must keep explicit control of the
logarithmic Hamiltonian generators arising from $\sigma^\#(dz_i/z_i)$ in simple normal crossings
charts. Moreover, at each step one must separate the genuinely Poisson-direction contributions from
the ambient logarithmic geometry of the pair $(X,D)$, in the sense of Saito's logarithmic forms and
vector fields \cite{Saito80}. This separation is also the conceptual source of a gauge ambiguity
specific to the Poisson context: Casimir functions are constant along $\fol_\sigma$ and thus behave
as leafwise constants, leaving a residual freedom given by Casimir-valued gauge transformations in the effective centralizer
in normal form statements. By contrast, on a holomorphic symplectic open set one has
$\fol_\sigma=T_X$, and the Poisson-direction normalization becomes compatible with the classical
regular-singular picture.

Our principal theorem establishes a Poisson Poincar\'e--Dulac normal form for logarithmic
Poisson-flat connections with fixed Euler--Poisson principal part (\Cref{thm:PPD-local}). We fix a Poisson-flat logarithmic model with commuting residue data and impose an effective non-resonance condition formulated in the quotient of the canonical simple-normal-crossings Hamiltonian generators by their boundary relations. The phrase fixed principal part is used in the regular-singular sense: the holomorphic remainder is residue-free and starts in positive boundary order. Under these hypotheses, any logarithmic Poisson-flat connection with this fixed Euler--Poisson principal part is holomorphically gauge equivalent to the model, and the resulting normal form is unique up to Casimir-valued gauge transformations in the effective centralizer (\Cref{thm:uniq}). This
yields a rigid local classification of Poisson-flat connection data compatible with $\sigma$ near
an SNC divisor, in the regime where the symplectic foliation has smaller (and possibly varying)
rank.

A complementary theme concerns the representation-theoretic encoding of monodromy in the presence
of boundary. The Poisson direction of flatness is intrinsically leafwise, whereas the meridional
topology created by $D$ is not, in general, visible within individual symplectic leaves. In order
to record simultaneously leafwise transport and boundary winding, we study a tangentially twisted
leafwise groupoid associated with $(\fol_\sigma,D)$, constructed by gluing the interior leafwise
fundamental groupoid to a canonical meridional isotropy along the collar of the real oriented
blow-up \cite{PopescuPampu25}. A key technical input on the collar is a leafwise log Poincar\'e lemma
which separates a closed leafwise $1$-form into an exact part plus meridional periods (Lemma~\ref{lem:leafwise-log-poincare}).
In the effectively non-resonant regime, the Poisson Poincar\'e--Dulac normal form
produces commuting residue matrices, and we show that the leafwise monodromy representation extends
to the twisted groupoid with an explicit meridional character determined by these residues. This
provides a Poisson analogue of the groupoid Riemann--Hilbert philosophy for logarithmic and Stokes
phenomena \cite{GLP18,Sabbah13}, adapted to the fact that flatness is imposed only along the
symplectic foliation, and compatible with the groupoid viewpoint on logarithmic connections
\cite{Bischoff20,Bischoff22}. The twisted groupoid should be read as the interior leafwise monodromy together with compatible meridional boundary data; we do not use a global Morita equivalence with the interior groupoid.

We then apply the normal form and the groupoid description to meromorphic Poisson modules and to
Poisson triples tangent to $\fol_\sigma$, obtaining explicit local models and sharp resonance
constraints. These applications dovetail with algebro-geometric aspects of Poisson modules,
including Morita-theoretic features \cite{Cor20}, and with the geometry of projective Poisson
varieties \cite{Pym18}. Throughout, the classical residue calculus for ordinary logarithmic
connections, and its variants, serves as a guiding analogy and a source of useful invariants
\cite{Ohtsuki82}.

Section~\ref{sec:rank2-triples} treats the rank-two case in a concrete form tailored to the normal-form and monodromy questions addressed above. We fix conventions for the Poisson (Lichnerowicz) differential and recall the basic notions of Hamiltonian vector fields, Casimirs, and Poisson vector fields, together with the interpretation of the corresponding low-degree Poisson cohomology groups. We then adopt a working definition of a meromorphic Poisson module as a vector
bundle equipped with a Poisson-direction connection with poles along a divisor, and we explain how,
after choosing a local trivialization in rank two, the data become an $\mathfrak{sl}_2$-valued
connection matrix whose Poisson-flatness is equivalent to an $\mathfrak{sl}_2$ Maurer--Cartan system,
in line with the Lie-algebroid description of Poisson modules \cite{Polishchuk97,LGV13}. The section
also presents an explicit local affine model associated with the diagonal quadratic component
$L(1,1,1,1)$ on $\mathbb{P}^3$, producing Poisson triples and trace-free rank-two Poisson
connections to which the Poisson Poincar\'e--Dulac normalization applies, thereby connecting the
abstract normal-form theory to concrete projective Poisson geometries \cite{Pym18} and to
Morita-theoretic features of Poisson modules \cite{Cor20}.
\section{Geometric and Poisson background}\label{sec:background}

In this section we fix notation and record the basic holomorphic Poisson geometry used throughout.
We also state standing hypotheses on $(X,\sigma,D)$ ensuring that all objects below are well-defined
and admit clean local models in codimension one (in particular, along the SNC divisor $D$).

\subsection{Holomorphic Poisson manifolds}

\begin{definition}\label{def:holPois}
A \emph{holomorphic Poisson manifold} is a complex manifold $X$ equipped with a holomorphic bivector field
$$
\sigma \in H^0\big(X,\wedge^2 T_X\big)
$$
such that the induced bracket on holomorphic functions
$$
\{f,g\}_\sigma := \sigma(df,dg)\qquad (f,g\in\OO_X)
$$
satisfies the Jacobi identity. Equivalently, the Schouten bracket satisfies $[\sigma,\sigma]=0$.
\end{definition}

The bivector $\sigma$ canonically equips the holomorphic cotangent bundle with a Lie algebroid structure.
Let $\sigma^\#:T_X^*\to T_X$ be the holomorphic bundle map
$$
\sigma^\#(\alpha) := \sigma(\alpha,\cdot).
$$
For holomorphic $1$-forms $\alpha,\beta\in\Omega_X^1$, the \emph{Koszul bracket} is
$$
[\alpha,\beta]_\sigma \;:=\;
\mathcal{L}_{\sigma^\#(\alpha)}\beta \;-\; \mathcal{L}_{\sigma^\#(\beta)}\alpha \;-\; d\big(\sigma(\alpha,\beta)\big).
$$
With this bracket and anchor, $\Omega_X^1$ is the sheaf of sections of a holomorphic Lie algebroid,
which we denote by $A_\sigma$.
For a holomorphic function $f\in\OO_X$, the associated \emph{Hamiltonian vector field} is
$$
X_f := \sigma^\#(df)\in T_X,
$$
so that $X_f(g)=\{f,g\}_\sigma$ for all $g\in\OO_X$.
The image sheaf
$$
\mathcal{F}_\sigma := \im(\sigma^\#)\subset T_X
$$
is an (in general singular) involutive holomorphic distribution; its integral manifolds are the
(possibly singular) \emph{symplectic leaves} of $\sigma$. On a smooth leaf $L$, the restriction of $\sigma$
is nondegenerate and inverts a holomorphic symplectic $2$-form $\omega_L$ on $L$.

\begin{definition}\label{def:regular-locus}
Let $X^\circ\subset X$ be the \emph{Poisson-regular locus}, i.e.\ the largest open subset on which
$\sigma^\#$ has locally constant rank. On $X^\circ$ the distribution $\mathcal{F}_\sigma$ is a regular
holomorphic foliation.
\end{definition}
If $\operatorname{rank}_\C(\sigma|_{X^\circ})=2m$, then every leaf in $X^\circ$ has complex dimension $2m$,
and carries a holomorphic symplectic form $\omega_L$.

\subsection{Divisors and logarithmic geometry}

Let $D\subset X$ be a divisor with simple normal crossings.
\begin{definition}\label{def:snc}
A divisor $D \subset X$ has \emph{simple normal crossings} (SNC) if for every point $p\in D$ there exist
local holomorphic coordinates $(z_1,\dots,z_n)$ centered at $p$ such that
$$
D = \{z_1\cdots z_r = 0\}
$$
for some $1\le r\le n$, and each local branch $\{z_i=0\}$ is smooth.
\end{definition}
Write $D=\bigcup_{i=1}^k D_i$ for the decomposition into irreducible components.

\begin{definition}\label{def:logTB}
The \emph{logarithmic tangent sheaf} $T_X(-\log D)$ is the subsheaf of $T_X$ consisting of holomorphic vector
fields tangent to $D$ (equivalently, preserving the ideal sheaf of $D$). Concretely, in an SNC chart
$(z_1,\dots,z_n)$ with $D=\{z_1\cdots z_r=0\}$ one has
$$
T_X(-\log D)
=
\big\langle
z_1\partial_{z_1},\,
\dots,\,
z_r\partial_{z_r},\,
\partial_{z_{r+1}},\,
\dots,\,
\partial_{z_n}
\big\rangle_{\OO_X}.
$$
Dually, the \emph{logarithmic cotangent sheaf} $\Omega^1_X(\log D)$ is locally generated by
$$
\frac{dz_1}{z_1},\,
\dots,\,
\frac{dz_r}{z_r},\,
dz_{r+1},\,
\dots,\,
dz_n.
$$
\end{definition}

When $D$ is SNC, both $T_X(-\log D)$ and $\Omega^1_X(\log D)$ are locally free $\OO_X$-modules of rank $\dim X$.

\subsection{Hypotheses}

Throughout the main parts of the paper we impose the following assumptions.

\begin{itemize}
\item[(H1)] $D=\bigcup_{i=1}^k D_i\subset X$ is an SNC divisor with smooth irreducible components $D_i$.

\item[(H2)] The Poisson bivector $\sigma$ has constant rank in codimension one. Let $2m$ be the generic
complex rank of $\sigma^\#$ on $X$. We assume that the degeneracy locus
$$
\Dgn(\sigma) := \{x\in X : \operatorname{rank}_\C(\sigma_x) < 2m\}
$$
has complex codimension at least $2$. Equivalently, $\sigma^m \in H^0(X,\wedge^{2m}T_X)$ vanishes in codimension at least $2$.

\item[(H3)] Each component $D_i$ is a \emph{Poisson hypersurface}, i.e.\ its ideal sheaf $\mathcal{I}_{D_i}\subset\OO_X$
is a Poisson ideal:
$$
\{\mathcal{I}_{D_i},\OO_X\}_\sigma \subset \mathcal{I}_{D_i}.
$$
Equivalently, $\sigma$ is logarithmically tangent to $D$ in the sense that
$$
\sigma^\#\big(\Omega_X^1(\log D)\big)\subset T_X(-\log D),
$$
or, in local SNC coordinates, Hamiltonian flows preserve each branch $\{z_i=0\}$ of $D$.
\end{itemize}

\begin{remark}\label{rem:H2H3}
We stress that we do \emph{not} assume that $(X,\sigma)$ is generically symplectic (i.e.\ $2m=\dim X$), nor that
$D$ is a degeneracy divisor for $\sigma$. Hypothesis \textup{(H2)} is only a codimension-one regularity condition
on the rank of $\sigma$, while \textup{(H3)} is the geometric condition that $D$ is preserved by Hamiltonian flows.
\end{remark}

\subsection{Local SNC charts and logarithmic Hamiltonians}\label{lem:logHam}

Fix an SNC chart $(z_1,\dots,z_n)$ with $D=\{z_1\cdots z_r=0\}$. For $1\le i\le r$, consider the logarithmic form
$dz_i/z_i \in \Omega_X^1(\log D)$.
Assume \textup{(H3)}. In this  chart the vector fields
$$
X_{\log z_i}\ :=\ \sigma^\# \left(\frac{dz_i}{z_i}\right)
\qquad (1\le i\le r)
$$
extend holomorphically across $D$ and belong to $T_X(-\log D)$. For $i>r$, the Hamiltonian vector fields
$X_{z_i}=\sigma^\#(dz_i)$ are holomorphic (with no logarithmic behavior along $D$).
In fact, fix $1\le i\le r$ and set $f=z_i$, so $D_i=\{f=0\}$ locally. By \textup{(H3)}, the ideal $(f)$ is Poisson, hence
for any holomorphic function $g$ there exists $h\in\OO_X$ such that
$$
\{f,g\}_\sigma = f h.
$$
On the other hand,
$$
\sigma^\# \left(\frac{df}{f}\right)(g)
=
\sigma \left(\frac{df}{f},dg\right)
=
\frac{1}{f}\{f,g\}_\sigma
=
h,
$$
which is holomorphic. Therefore $\sigma^\#(df/f)$ defines a holomorphic vector field on the chart.
Moreover, since $(f)$ is preserved, this vector field is tangent to $D_i$, and since this holds for each local branch,
it is tangent to $D$; hence it lies in $T_X(-\log D)$.
For $i>r$ there is no pole to begin with, so $\sigma^\#(dz_i)$ is holomorphic.

From now on, whenever we choose local coordinates $(z_1,\ldots,z_n)$ with $D=\{z_1\cdots z_r=0\}$, we implicitly
work in an SNC chart and use the logarithmic Hamiltonians $X_{\log z_i}=\sigma^\#(dz_i/z_i)$ along the components of $D$.

\section{Logarithmic Poisson connections}\label{sec:log-poisson-connections}
\begin{lemma}\label{lem:log-cotangent-lie}
Assume \textup{(H3)}. Then the Koszul bracket $[-,-]_\sigma$ preserves logarithmic forms:
\[
[\Omega_X^1(\log D),\,\Omega_X^1(\log D)]_\sigma \subset \Omega_X^1(\log D).
\]
Equivalently, $A_\sigma(\log D):=\Omega_X^1(\log D)$ is a holomorphic Lie algebroid with anchor
$\rho=\sigma^\#:\Omega_X^1(\log D)\to T_X(-\log D)$.
\end{lemma}

\begin{proof}
Work in an SNC chart $(z_1,\dots,z_n)$ with $D=\{z_1\cdots z_r=0\}$.
It suffices to check the generators $dz_i/z_i$ ($i\le r$) and $dz_j$ ($j>r$).
By \ref{lem:logHam}, the vector fields $X_i:=\sigma^\#(dz_i/z_i)$ lie in $T_X(-\log D)$.
Hence $\mathcal L_{X_i}$ preserves $\Omega_X^1(\log D)$.
Now, compute the Koszul bracket
\[
[\alpha,\beta]_\sigma=\mathcal L_{\sigma^\#(\alpha)}\beta-\mathcal L_{\sigma^\#(\beta)}\alpha-d(\sigma(\alpha,\beta)).
\]
The first two Lie-derivative terms are logarithmic because $\sigma^\#(\alpha)$ and $\sigma^\#(\beta)$ are logarithmic
vector fields and Lie derivatives by logarithmic vector fields preserve $\Omega_X^1(\log D)$.
For the last term, $\sigma(\alpha,\beta)$ is a holomorphic function on the chart (since $\sigma^\#(\Omega^1(\log D))$
is holomorphic and $\alpha,\beta$ have at most logarithmic poles), so $d(\sigma(\alpha,\beta))$ is holomorphic, hence
in $\Omega_X^1\subset\Omega_X^1(\log D)$.
Therefore $[\alpha,\beta]_\sigma$ is logarithmic.
\end{proof}

We now define the central object of study: a \emph{representation} of the logarithmic cotangent
Poisson Lie algebroid. Concretely, we consider
$$
A_\sigma(\log D):=\Omega_X^1(\log D),
$$
equipped with the Koszul bracket $[-,-]_\sigma$ and the anchor map
$$
\rho:=\sigma^\#:\Omega_X^1(\log D)\longrightarrow T_X(-\log D),
$$
which lands in $T_X(-\log D)$ by hypothesis (H3).
We set
$$
\fol_\sigma^{\log}
:=
\im \big(\sigma^\#:\Omega_X^1(\log D)\to T_X(-\log D)\big)
\ \subset\ T_X(-\log D).
$$
Over $X\setminus D$ one has $\Omega_X^1(\log D)=\Omega_X^1$, hence $\fol_\sigma^{\log}=\fol_\sigma$ there.
We write $\mathfrak X_\sigma^{\log}(U):= H^0(U,\fol_\sigma^{\log})$ for the holomorphic logarithmic Hamiltonian
vector fields on $U$, and $\mathfrak X_\sigma(U):= H^0(U\setminus D,\fol_\sigma)$ on $U\setminus D$.

\subsection{Definition and basic properties}

Let $E\to X$ be a holomorphic vector bundle of rank $e$, and write $\OO_X(E)$ for its sheaf of holomorphic sections.
Let $d_{A_\sigma}$ denote the Chevalley--Eilenberg differential of the Lie algebroid $A_\sigma(\log D)$.
Since $A_\sigma(\log D)^\vee\simeq T_X(-\log D)$, the Lie algebroid cochain complex is
$$
 H^0\big(X,\wedge^\bullet A_\sigma(\log D)^\vee\big)
\;=\;
 H^0\big(X,\wedge^\bullet T_X(-\log D)\big).
$$
In particular, on functions $f\in\OO_X$ we obtain a section $d_{A_\sigma}f\in H^0\big(X,T_X(-\log D)\big)$
characterized by
$$
\alpha\big(d_{A_\sigma}f\big)=\rho(\alpha)(f)
\qquad\text{for all }\alpha\in\Omega_X^1(\log D).
$$
Equivalently, if we set $X_f:=\sigma^\#(df)$ (the Hamiltonian vector field), then
\begin{equation}\label{eq:dA-on-f}
d_{A_\sigma}f \;=\; X_f \;=\; \sigma^\#(df).
\end{equation}
To simplify notation, we write $\delta:=d_{A_\sigma}$ for the Chevalley--Eilenberg differential of the logarithmic cotangent Poisson Lie algebroid
$A_\sigma(\log D)=\Omega_X^1(\log D)$. Under the standard identification
$A_\sigma(\log D)^\vee\simeq T_X(-\log D)$, this is the same differential induced by the Poisson/Lichnerowicz differential $[\sigma,-]$ on polyvector fields.
In particular, for $f\in\OO_X$ we fix the convention
\begin{equation}\label{eq:delta-conv}
\delta(f):=\sigma^\#(df)=:X_f\in T_X(-\log D).
\end{equation}
When $\delta$ is applied to a matrix of functions, it is understood entrywise.

\begin{definition}[Logarithmic Poisson connection]\label{def:logPois}
A \emph{logarithmic Poisson connection tangent to the symplectic foliation} is a $\C$-linear morphism of sheaves
\[
\begin{aligned}
\nabla^{\mathrm{Pois}}:\OO_X(E)&\longrightarrow \fol_\sigma^{\log}\otimes_{\OO_X}\OO_X(E) \subset T_X(-\log D)\otimes_{\OO_X}\OO_X(E) 
\end{aligned}
\]
satisfying the Leibniz rule
\begin{equation}\label{eq:Leibniz}
\nabla^{\mathrm{Pois}}(f s)
\;=\;
\delta(f)\otimes s \;+\; f\,\nabla^{\mathrm{Pois}}(s)
\end{equation}
for all local holomorphic functions $f$ and local holomorphic sections $s$ of $E$.
\end{definition}
Given $\nabla^{\mathrm{Pois}}$ one may define a $\C$-bilinear bracket
$$
\{\,\cdot\,,\,\cdot\,\}:\OO_X\times \OO_X(E)\to \OO_X(E),
\qquad
\{f,s\}:=\big(\nabla^{\mathrm{Pois}}(s)\big)(df),
$$
i.e.\ contraction of the $T_X(-\log D)$-valued section $\nabla^{\mathrm{Pois}}(s)$ with the $1$-form $df\in\Omega_X^1\subset\Omega_X^1(\log D)$.
This bracket records the action of the connection on ordinary exact forms. In the logarithmic setting, however, the operators attached to the logarithmic generators $dz_i/z_i$ are part of the connection data and are not determined by holomorphic functions alone. Thus the residue terms below are genuine logarithmic Lie-algebroid data. Flatness of $\nabla^{\mathrm{Pois}}$ is equivalently the vanishing of the Lie-algebroid curvature, and on ordinary exact forms it gives the Jacobi-type identity
$$
\{\{f,g\},s\}=\{f,\{g,s\}\}-\{g,\{f,s\}\}.
$$

Recall that an ordinary logarithmic connection is a morphism $$\nabla^{\log}:E\to\Omega_X^1(\log D)\otimes E$$
satisfying $\nabla^{\log}(fs)=df\otimes s+f\nabla^{\log}(s)$.
A logarithmic Poisson connection is \emph{not} obtained from $\nabla^{\log}$ in general,  because its Leibniz rule is governed by
the Poisson derivation $\delta(f)=\sigma^\#(df)$ rather than $df$.
However, on any open set where the logarithmic anchor
$$\sigma^\#:\Omega_X^1(\log D)\to T_X(-\log D)$$ is an isomorphism (e.g.\ in the log-symplectic/reduced-degeneracy case),
there is a one-to-one correspondence
$$
\nabla^{\mathrm{Pois}}
\longleftrightarrow
\nabla^{\log}:=(\sigma^\#)^{-1}\circ \nabla^{\mathrm{Pois}},
$$
and flatness is preserved under this identification. If $D$ is not reduced, the same mechanism yields a
\emph{meromorphic} flat connection with poles along $D$.

Now, fix an open set $U\subset X$ over which $E$ admits a holomorphic frame $e=(e_1,\dots,e_e)$.
Write
\begin{equation}\label{eq:Theta-loc}
\nabla^{\mathrm{Pois}}(e_a)
\;=\;
\sum_{b=1}^e \Theta_{ba}\otimes e_b,
\end{equation}
where each $\Theta_{ba}$ is a local section of $\fol_\sigma^{\log}|_U$.
Equivalently, $\Theta=(\Theta_{ba})$ is an $e\times e$ matrix with entries in $\fol_\sigma^{\log}$,
 and
$$
\nabla^{\mathrm{Pois}}(e)=\Theta\cdot e.
$$

For a general section $s=\sum_a Y^a e_a$ (with coefficient column vector $Y=(Y^1,\dots,Y^e)^T$),
the Leibniz rule \eqref{eq:Leibniz} gives
\begin{equation}\label{eq:connection-on-section}
\nabla^{\mathrm{Pois}}(s)
=
\big(\delta Y + \Theta\,Y\big)\otimes e,
\end{equation}
where $\delta Y$ denotes the column vector whose $a$-th entry is $\delta(Y^a)\in T_X(-\log D)$.
Thus the horizontality condition $\nabla^{\mathrm{Pois}}(s)=0$ is equivalently
\begin{equation}\label{eq:horizontal}
\delta Y \;=\; -\,\Theta\,Y,
\end{equation}
an equality in $T_X(-\log D)^{\oplus e}$.

\begin{remark}[Gauge transformation]\label{rem:gauge}
Under a change of frame $e'=g\,e$ with $g\in GL_e(\OO_U)$, the connection matrix transforms as
$$
\Theta'=(\delta g)\,g^{-1}+g\,\Theta\,g^{-1},
$$
where $\delta g$ is obtained by applying $\delta$ entrywise to the matrix $g$.
\end{remark}
View $\Theta$ as an element of
$$
 H^0\big(U, A_\sigma(\log D)^\vee\otimes \End(E)\big)
=
 H^0\big(U, T_X(-\log D)\otimes \End(E)\big),
$$
i.e.\ an $\End(E)$-valued $1$-cochain in the Lie algebroid complex of $A_\sigma(\log D)$.
The differential $\delta$ extends to $\End(E)$-valued cochains by the usual Leibniz rule.
Define the wedge--composition product by
$$
(\Theta\wedge\Theta)(\alpha,\beta)
:=\Theta(\alpha)\circ\Theta(\beta)-\Theta(\beta)\circ\Theta(\alpha),
\qquad
\alpha,\beta\in\Omega_X^1(\log D).
$$

\begin{definition}[Poisson curvature / flatness]\label{def:flatness}
The \emph{Poisson curvature} of $\nabla^{\mathrm{Pois}}$ in the trivialization $e$ is the $\End(E)$-valued $2$-cochain
\begin{equation}\label{eq:curvature}
\mathcal{K}^{\mathrm{Pois}}
\;:=\;
\delta\Theta \;+\; \Theta \wedge \Theta
\;\in\;
 H^0\big(U,\wedge^2 T_X(-\log D)\otimes \End(E)\big).
\end{equation}
We say that $\nabla^{\mathrm{Pois}}$ is \emph{Poisson-flat} (or simply \emph{flat}) if
\begin{equation}\label{eq:flat}
\delta\Theta + \Theta\wedge\Theta \;=\; 0.
\end{equation}
\end{definition}

\begin{remark}[Maurer--Cartan form]\label{rem:MC}
Equation \eqref{eq:flat} is the Maurer--Cartan equation in the differential graded algebra
$\big( H^0(U,\wedge^\bullet A_\sigma(\log D)^\vee)\otimes \End(E),\,\delta,\,\wedge\big)$.
The Jacobi identity for the Koszul bracket (equivalently $[\sigma,\sigma]=0$) is what ensures $\delta^2=0$
and hence makes \eqref{eq:curvature} intrinsic and gauge-invariant.
\end{remark}

\section{Residues and Euler--Poisson principal parts}\label{sec:residues}

We work near a point $p\in D\cap X^\circ$ and choose an adapted SNC chart
$U\ni p$ with holomorphic coordinates $(z_1,\dots,z_n)$ such that
$D\cap U=\{z_1\cdots z_r=0\}$ and $D_i\cap U=\{z_i=0\}$ for $1\le i\le r$.
For $1\le i\le r$ set
$$
X_i \ :=\ \sigma^\# \Big(\frac{dz_i}{z_i}\Big)\ \in\ H^0\big(U,\,T_X(-\log D)\big).
$$
By \ref{lem:logHam} these are holomorphic logarithmic vector fields.
Assume \textup{(H3)}. For each $1\le i\le r$, the vector field
$X_i=\sigma^\#(dz_i/z_i)$ is a Poisson vector field, i.e.
\begin{equation}\label{lem:logHam-poisson}
X_i\big|_{V}\ =\ \sigma^\#\big(d(\log z_i)\big)\ =\ \delta(\log z_i).
\end{equation}

\subsection{Euler--Poisson principal parts and residues}

\begin{definition}[Euler--Poisson principal part and residue]\label{def:EP-residue}
Fix commuting constant matrices $A_1,\dots,A_r\in\Mat_{e\times e}(\C)$.
The associated \emph{Euler--Poisson principal part} on $U$ is the $\End(E)$-valued Poisson connection matrix
$$
\Theta_0\ :=\ \sum_{i=1}^r A_i\,X_i
\ =\ \sum_{i=1}^r A_i\,\sigma^\# \Big(\frac{dz_i}{z_i}\Big).
$$
We define the \emph{residue of $\Theta_0$ along $D_i$ (at $p$)} to be
$$
\Res^{\mathrm{Pois}}_{D_i,p}(\Theta_0)\ :=\ A_i.
$$
\end{definition}

\begin{proposition}[Flatness of Euler--Poisson principal parts]\label{prop:EP-flat}
Let
$$
\Theta_0=\sum_{i=1}^r A_iX_i,
\qquad
X_i=\sigma^\# \Big(\frac{dz_i}{z_i}\Big),
$$
with constant matrices $A_i\in\Mat_{e\times e}(\C)$. If the residue tuple commutes, i.e.
$$
[A_i,A_j]=0\qquad\text{for all }i,j,
$$
then $\Theta_0$ is Poisson-flat. Conversely, if the bivectors $X_i\wedge X_j$, $1\le i<j\le r$, are linearly independent outside an analytic subset of codimension at least two, then flatness of $\Theta_0$ forces $[A_i,A_j]=0$ for all $i,j$.
\end{proposition}

\begin{proof}
Since each $X_i$ is Poisson (\ref{lem:logHam-poisson}) and each $A_i$ is constant,
we have $\delta\Theta_0=0$. Moreover,
$$
\Theta_0\wedge\Theta_0
=
\sum_{1\le i<j\le r} [A_i,A_j]\; X_i\wedge X_j.
$$
Thus a commuting residue tuple makes the curvature vanish. Conversely, under the stated independence hypothesis, the vanishing of this holomorphic section outside the codimension-two subset forces each coefficient $[A_i,A_j]$ to vanish there, and hence everywhere by holomorphic extension.
\end{proof}

\subsection{Leafwise meridians and monodromy  }

Shrink $U$ so that $U\subset X^\circ$ and set $U^*:=U\setminus D$.
Assume $\nabla^{\mathrm{Pois}}$ is Poisson-flat on $U^*$ and that we have chosen an Euler--Poisson gauge on $U$,
so that its matrix is $\Theta_0=\sum_{i=1}^r A_i X_i$ with commuting constant residues.
Fix $x\in U^*$ and let $L_x$ be the symplectic leaf through $x$.
On $L_x\cap U^*$ the Poisson connection induces an ordinary flat logarithmic connection, and for a small meridian
$\gamma_i$ around $D_i\cap L_x$ one has the classical formula
$$
\rho^{\mon}_{L_x,x}([\gamma_i])\ =\ \exp\big(-2\pi i\,A_i\big),
$$
with the sign matching the convention used in the transport equation.
This shows that, in Euler--Poisson gauge, meridional monodromy is determined by the residue tuple.

\section{Non-resonance and the Poisson--Poincar\'e--Dulac normal form}\label{sec:PPD}

In this section we work locally near a point $p\in D\cap X^\circ$ and restrict to the effectively non-resonant regime.  The local statement is deliberately formulated with a fixed Euler--Poisson principal part.  This is not an additional weakening of the problem, but the regular-singular condition that the logarithmic residue data are already prescribed.  Thus the holomorphic remainder is required to be residue-free, and the induction is carried out with respect to the boundary ideal of the SNC stratum.

Fix an adapted chart $U\ni p$ and a holomorphic frame $e$.  Let $\Theta$ be the local matrix of $\nabla^{\mathrm{Pois}}$ on $U$.  We write
\begin{equation}\label{eq:Theta-split}
\Theta=\Theta_0+V,
\qquad
\Theta_0:=\sum_{i=1}^r A_i^\circ\,\sigma^\#\Big(\frac{dz_i}{z_i}\Big)=\sum_{i=1}^r A_i^\circ X_i,
\end{equation}
where $A_i^\circ\in\Mat_{e\times e}(\C)$ form a commuting residue tuple.  By Proposition~\ref{prop:EP-flat}, $\Theta_0$ is Poisson-flat.

\begin{notation} \label{not:boundary-ideal}
In the chosen SNC chart$
D\cap U=\{z_1\cdots z_r=0\},
$
we set
\[
I_D:=(z_1,\ldots,z_r)\subset\OO_U,
\quad
X_i:=\sigma^\#\Big(\frac{dz_i}{z_i}\Big),\quad 1\le i\le r.
\]
For $\alpha=(\alpha_1,\ldots,\alpha_r)\in\ZZ_{\ge0}^r$ we write
$
z^\alpha:=z_1^{\alpha_1}\cdots z_r^{\alpha_r}$, $
  |\alpha|:=\alpha_1+\cdots+\alpha_r.$
Then
\begin{equation}\label{eq:delta-monomial}
\delta(z^\alpha)=\sigma^\#(d z^\alpha)
=z^\alpha\sum_{i=1}^r \alpha_i X_i.
\end{equation}
The variables $z_{r+1},\ldots,z_n$ are treated as holomorphic transverse parameters in the $I_D$-adic expansion.
\end{notation}

\begin{definition}[Fixed Euler--Poisson principal part]\label{def:fixed-EP-principal-part}
Let $\Theta_0=\sum_i A_i^\circ X_i$ be as in \eqref{eq:Theta-split}.  We say that a logarithmic Poisson connection matrix $\Theta$ has \emph{fixed Euler--Poisson principal part} $\Theta_0$ if
\[
\Theta=\Theta_0+V
\]
and the remainder $V$ is residue-free and positive with respect to the boundary ideal.  Concretely, in the chosen Euler--Poisson chart this means that the component of $V$ in each logarithmic meridional generator $X_i$ has vanishing residue along $D_i$ and that, after the standard local residue-free normalization, one has
\begin{equation}\label{eq:residue-free-positive}
V\in I_D\cdot \Mat_{e\times e}\big(\fol_\sigma^{\log}(U)\big).
\end{equation}
Equivalently, $\Theta$ and $\Theta_0$ have the same Euler--Poisson residues along all local branches of $D$, and the remaining part begins in strictly positive boundary order.  If the generators $X_i$ satisfy relations, the condition is imposed modulo their effective boundary relations, as made precise by the Euler--Poisson quotient introduced below; the residue tuple is then understood only through the induced effective principal part.
\end{definition}

\begin{remark}\label{rem:fixed-principal-part-not-weakening}
Condition \eqref{eq:residue-free-positive} is the Poisson-logarithmic analogue of fixing the polar part of a regular-singular logarithmic connection.  A merely holomorphic remainder need not have zero logarithmic residue; for instance a coefficient of $dz_i/z_i$ may be holomorphic and still have non-zero restriction to $D_i$.  Thus the residue-free condition is the intrinsic meaning of having the same principal part.
\end{remark}

\subsection{Joint generalized eigenspaces and effective non-resonance}

Fix the joint generalized eigenspace decomposition of the commuting residue tuple
$(A_1^\circ,\ldots,A_r^\circ)$:
\[
        \C^e=\bigoplus_{\kappa\in\mathcal I}E_\kappa .
\]
Thus, for every $\kappa\in\mathcal I$ and $1\le i\le r$,
\[
        A_i^\circ|_{E_\kappa}=\lambda_{i,\kappa}\Id_{E_\kappa}+N_{i,\kappa},
\]
where $\lambda_{i,\kappa}\in\C$ and $N_{i,\kappa}$ is nilpotent.  We write
$\lambda_\kappa=(\lambda_{1,\kappa},\ldots,\lambda_{r,\kappa})\in\C^r$.

\begin{definition}\label{def:effective-ep-quotient}
Let
\[
        X_i=\sigma^\#\!\left(\frac{dz_i}{z_i}\right),
        \qquad 1\le i\le r,
\]
be the logarithmic Hamiltonian generators in the adapted SNC chart.  Let
$\mathfrak k_{\sigma,p}\subset\C^r$ be the space of constant linear relations
among their leading boundary classes:
\[
        \mathfrak k_{\sigma,p}
        =
        \left\{
        (c_1,\ldots,c_r)\in\C^r
        \ ;\
        \sum_{i=1}^r c_iX_i
        \equiv
        0
        \pmod {I_D\fol^{\log}_\sigma}
        \right\}.
\]
The effective Euler--Poisson space at $p$ is
$\mathfrak t_{\sigma,p}:=\C^r/\mathfrak k_{\sigma,p}$.  If
$q_{\sigma,p}:\C^r\to\mathfrak t_{\sigma,p}$ is the quotient map and
$\mu=(\mu_1,\ldots,\mu_r)\in\C^r$, set
$[\mu]_\sigma:=q_{\sigma,p}(\mu)$.  Thus $[\mu]_\sigma$ represents the effective
boundary vector $\sum_i\mu_iX_i$, modulo the leading relations among the
$X_i$.
The effective weight of $E_\kappa$ is
\[
        \bar\lambda_\kappa:=[\lambda_\kappa]_\sigma
        \in\mathfrak t_{\sigma,p}.
\]
Equivalently, changing the residue tuple by coefficients lying in
$\mathfrak k_{\sigma,p}$ does not change the Euler--Poisson principal part
$\Theta_0=\sum_iA_i^\circ X_i$.  If the logarithmic Hamiltonians
$X_1,\ldots,X_r$ are independent modulo $I_D\fol_\sigma^{\log}$, then
$\mathfrak t_{\sigma,p}=\C^r$.
\end{definition}

\begin{definition}\label{def:nonres}
The Euler--Poisson principal part $\Theta_0$ is effectively
Poisson--Poincar\'e--Dulac non-resonant at $p$ if, for every
$\kappa\neq\kappa'$ and every non-zero $\alpha\in\ZZ_{\ge0}^r$,
\[
        [\alpha]_\sigma
        +
        \bar\lambda_\kappa
        -
        \bar\lambda_{\kappa'}
        \neq
        0
        \quad\text{in }\mathfrak t_{\sigma,p}.
\]
For the analytic normalization theorem we also require the standard Poincar\'e
lower bound for these non-zero effective vectors, as needed in the majorant
argument.  In the independent-generator case $\mathfrak t_{\sigma,p}=\C^r$, the
condition is
\[
        \alpha+\lambda_\kappa-\lambda_{\kappa'}\neq0
        \quad\text{in }\C^r .
\]
\end{definition}

\subsection{Leafwise logarithmic primitives and the homological equation}

\begin{lemma}\label{lem:leafwise-log-poincare}
Let $U\simeq [0,\varepsilon)^r\times(S^1)^r\times B$ be a local collar chart
of the oriented real blow-up, with angular coordinates
$\theta_1,\ldots,\theta_r$.  Let $\fol$ be a smooth foliation on $U$ such that,
after shrinking $U$, each leaf is of the form
\[
        L\simeq (S^1)^r\times\mathbf D^q,
\]
and the functions $\theta_i$ restrict to the standard circle coordinates on
the $(S^1)^r$-factor.  If $\alpha$ is a smooth leafwise $1$-form with
$d_\fol\alpha=0$, then, after shrinking $U$ if necessary,
\[
        \alpha=d_\fol f+\sum_{i=1}^r c_i\,d\theta_i,
\]
where $f$ is a smooth leafwise function and the $c_i$ are constant along the
leaves.  Moreover,
\[
        c_i|_L
        =
        \frac{1}{2\pi}\int_{\gamma_i}\alpha|_L,
\]
where $\gamma_i\subset L$ is the standard loop in the $i$-th circle factor.
In particular, $\alpha$ is leafwise exact if and only if all these meridional
periods vanish.
\end{lemma}

\begin{proof}
On each leaf $L\simeq(S^1)^r\times\mathbf D^q$, the classes of
$d\theta_1,\ldots,d\theta_r$ form a basis of $H^1(L;\mathbb R)$.  Define $c_i$
by the period formula above.  Then $\alpha-\sum_i c_i\,d\theta_i$ has zero
periods on all generators of $H_1(L;\mathbb Z)$, hence is exact on $L$.  After
shrinking the collar, the primitives may be chosen smoothly in the transverse
parameters.  The last assertion follows from the same period description.
\end{proof}

\begin{lemma}\label{lem:holomorphic-residue-free-primitive}
Let $B\subset\C^m$ be a polydisc, let $D_B=\{z_1\cdots z_r=0\}$, and set
$I_D=(z_1,\ldots,z_r)$.  Let
\[
        \eta
        =
        \sum_{i=1}^r a_i(z)\frac{dz_i}{z_i}
        +
        \sum_{j=r+1}^m b_j(z)\,dz_j,
        \qquad
        a_i,b_j\in\OO(B),
\]
be a closed holomorphic logarithmic $1$-form on $B\setminus D_B$.  If
$a_i|_{z_i=0}=0$ for $1\le i\le r$, then $\eta=df$ for some $f\in\OO(B)$.
If, moreover, all coefficients $a_i,b_j$ lie in $I_D^N$, then $f$ may be
chosen in $I_D^N$.
\end{lemma}

\begin{proof}
The condition $a_i|_{z_i=0}=0$ gives $a_i=z_i a_i'$, with $a_i'\in\OO(B)$.
Hence $a_i(z)dz_i/z_i=a_i'(z)\,dz_i$ is holomorphic.  Thus $\eta$ extends to a
holomorphic closed $1$-form on the whole polydisc.  The holomorphic
Poincar\'e lemma gives a holomorphic primitive.  The straight-line homotopy
formula preserves the $I_D^N$-order, so the primitive can be chosen in
$I_D^N$.
\end{proof}

\begin{lemma}\label{lem:residue-free-EP-acyclicity}
Let $L:=\delta+[\Theta_0,-]$, where $\Theta_0=\sum_iA_i^\circ X_i$, and set
$Q_N:=I_D^N/I_D^{N+1}$.  Let
$\eta\in Q_N\otimes\Mat_{e\times e}(\fol_\sigma^{\log})$ satisfy
\[
        L\eta=0
        \quad\text{in }Q_N\otimes\wedge^2\fol_\sigma^{\log},
        \qquad
        \Res_{D_i}(\eta)=0
        \quad(1\le i\le r).
\]
Then, after shrinking the adapted chart if necessary, there exists
$K\in\Mat_{e\times e}(Q_N)$ such that $LK=-\eta$.
\end{lemma}

\begin{proof}
Decompose the proof according to $\C^e=\bigoplus_\kappa E_\kappa$.  For
$W_{\kappa\kappa'}:=\Hom(E_{\kappa'},E_\kappa)$, put
\[
        \mathcal N_i^{\kappa\kappa'}(M)
        =
        N_{i,\kappa}M-MN_{i,\kappa'} .
\]
Since the $\mathcal N_i^{\kappa\kappa'}$ are commuting nilpotent endomorphisms,
choose a finite filtration
\[
        W_{\kappa\kappa'}=F^0W_{\kappa\kappa'}
        \supset F^1W_{\kappa\kappa'}
        \supset\cdots\supset F^QW_{\kappa\kappa'}=0
\]
such that $\mathcal N_i^{\kappa\kappa'}(F^qW_{\kappa\kappa'})\subset
F^{q+1}W_{\kappa\kappa'}$ for all $i$.  It is enough to prove exactness on the
associated graded.
The graded boundary piece decomposes as
\[
        Q_N=\bigoplus_{|\alpha|=N}z^\alpha\OO_S,
        \qquad
        S=\{z_1=\cdots=z_r=0\}.
\]
On the $z^\alpha$-summand and on the block $(\kappa,\kappa')$, the semisimple
part of $L$ is multiplication, in the Koszul complex of the effective space, by
\[
        \chi_{\alpha,\kappa,\kappa'}
        =
        [\alpha]_\sigma+\bar\lambda_\kappa-\bar\lambda_{\kappa'}
        \in\mathfrak t_{\sigma,p}.
\]
If $\chi_{\alpha,\kappa,\kappa'}\neq0$, choose
$\ell\in\mathfrak t_{\sigma,p}^{\vee}$ with
$\ell(\chi_{\alpha,\kappa,\kappa'})=1$.  For
$d_\chi:=\chi_{\alpha,\kappa,\kappa'}\wedge(-)$, contraction
$h_\ell:=\iota_\ell$ satisfies
\[
        d_\chi h_\ell+h_\ell d_\chi=\Id .
\]
Hence the associated graded complex is exact in degree $1$.
If $\chi_{\alpha,\kappa,\kappa'}=0$, the associated graded equation is
$\delta K=-\eta$.  The equation $L\eta=0$ gives closedness of the corresponding
logarithmic leafwise $1$-form, and $\Res_{D_i}(\eta)=0$ gives zero logarithmic
residues.  Lemma~\ref{lem:holomorphic-residue-free-primitive} therefore gives
a holomorphic primitive of the same $I_D$-order.
Thus the associated graded complex is exact in degree $1$ in both cases.  The
finite filtration above lifts the primitive from the associated graded to the
original filtered complex.  Hence $LK=-\eta$ for some
$K\in\Mat_{e\times e}(Q_N)$.
\end{proof}

\begin{lemma}\label{lem:EP-homological-equation}
Let $R^{(N)}$ be the order $N$ defect term produced in the $I_D$-adic
normalization of a flat logarithmic Poisson connection with fixed
Euler--Poisson principal part $\Theta_0$.  Suppose that $\Theta_0$ satisfies
Definition~\ref{def:nonres}.  Then, after shrinking the adapted chart if
necessary, there exists $K^{(N)}\in\Mat_{e\times e}(I_D^N)$ such that
\[
        \delta K^{(N)}+[\Theta_0,K^{(N)}]
        \equiv
        -R^{(N)}
        \pmod {I_D^{N+1}} .
\]
\end{lemma}

\begin{proof}
Set $L:=\delta+[\Theta_0,-]$.  Since $\Theta_0$ is flat, $L^2=0$.  The order
$N$ part of the Maurer--Cartan equation gives
\[
        L(R^{(N)})\equiv0\pmod {I_D^{N+1}} .
\]
It is enough to solve the equation in $Q_N:=I_D^N/I_D^{N+1}$, with
\[
        Q_N=\bigoplus_{|\alpha|=N}z^\alpha\OO_S,
        \qquad
        S=\{z_1=\cdots=z_r=0\}.
\]

Consider the $(\kappa,\kappa')$-block with $\kappa\neq\kappa'$.  On the
$z^\alpha$-summand the effective vector is
\[
        \chi_{\alpha,\kappa,\kappa'}
        =
        [\alpha]_\sigma+\bar\lambda_\kappa-\bar\lambda_{\kappa'}.
\]
Definition~\ref{def:nonres} gives $\chi_{\alpha,\kappa,\kappa'}\neq0$.  Choose
$\ell\in\mathfrak t_{\sigma,p}^{\vee}$ such that
$\ell(\chi_{\alpha,\kappa,\kappa'})=1$.  The Koszul homotopy
\[
        \bigl[\chi_{\alpha,\kappa,\kappa'}\wedge(-)\bigr]\iota_\ell
        +
        \iota_\ell\bigl[\chi_{\alpha,\kappa,\kappa'}\wedge(-)\bigr]
        =
        \Id
\]
solves the semisimple off-diagonal meridional equation.  The nilpotent
contribution is treated by the finite filtration used in the proof of
Lemma~\ref{lem:residue-free-EP-acyclicity}.

Summing these off-diagonal meridional corrections over all $\alpha$ and all
$\kappa\neq\kappa'$, we obtain
$K_{\mathrm{mer}}^{(N)}\in\Mat_{e\times e}(Q_N)$.  Define
\[
        R_{\mathrm{rf}}^{(N)}=R^{(N)}+L(K_{\mathrm{mer}}^{(N)}).
\]
By construction, the off-diagonal logarithmic residues of
$R_{\mathrm{rf}}^{(N)}$ vanish.  Since $L^2=0$, the class
$R_{\mathrm{rf}}^{(N)}$ is still $L$-closed in $Q_N$.  On diagonal blocks, the
fixed-principal-part hypothesis gives vanishing logarithmic residues.  Hence
$R_{\mathrm{rf}}^{(N)}$ is residue-free and $L$-closed.
Lemma~\ref{lem:residue-free-EP-acyclicity} gives
$K_{\mathrm{rf}}^{(N)}\in\Mat_{e\times e}(Q_N)$ such that
$L(K_{\mathrm{rf}}^{(N)})=-R_{\mathrm{rf}}^{(N)}$ in $Q_N$.  Therefore
$K_Q^{(N)}:=K_{\mathrm{mer}}^{(N)}+K_{\mathrm{rf}}^{(N)}$ satisfies
$L(K_Q^{(N)})=-R^{(N)}$ in $Q_N$.  Choose any lift
$K^{(N)}\in\Mat_{e\times e}(I_D^N)$ of $K_Q^{(N)}$.  This lift satisfies the
desired congruence modulo $I_D^{N+1}$.
\end{proof}

\subsection{Poisson--Poincar\'e--Dulac elimination in the effectively non-resonant case}

Under a holomorphic gauge $e\mapsto \widetilde e=He$, the connection matrix transforms as
\[
\widetilde\Theta=(\delta H)H^{-1}+H\Theta H^{-1}.
\]

\begin{theorem} \label{thm:PPD-local}
Let $p\in D\cap X^\circ$, and work in an adapted Euler--Poisson chart as above.  Assume that $\nabla^{\mathrm{Pois}}$ is Poisson-flat, that $\Theta$ has fixed Euler--Poisson principal part $\Theta_0=\sum_i A_i^\circ X_i$ in the sense of Definition~\ref{def:fixed-EP-principal-part}, and that the Euler--Poisson principal part is effectively PPD non-resonant at $p$ in the sense of Definition~\ref{def:nonres}.  Then, after shrinking $U$ around $p$, there exists a holomorphic gauge $H\in\GL_e(\OO_U)$ such that
\begin{equation}\label{eq:PPD-Theta}
\widetilde\Theta=(\delta H)H^{-1}+H\Theta H^{-1}
=\sum_{i=1}^r \widetilde A_i\,\sigma^\#\Big(\frac{dz_i}{z_i}\Big),
\end{equation}
where the matrices $\widetilde A_i\in\Mat_{e\times e}(\C)$ are constant and commuting.  The normal form is written relative to the chosen Euler--Poisson generators $X_i$; if these generators satisfy relations, the residue tuple is understood in the effective quotient $\mathfrak t_{\sigma,p}$.
\end{theorem}

\begin{proof}
We give the proof in the chosen effective PPD chart.  Since \(\Theta\) has fixed Euler--Poisson principal part, write
\[
\Theta=\Theta_0+V,
\qquad
\Theta_0=\sum_{i=1}^r A_i^\circ X_i,
\qquad
V\in I_D\cdot\Mat_{e\times e}(\fol_\sigma^{\log}(U)).
\]
The matrices \(A_i^\circ\) commute, hence \(\Theta_0\) is Poisson-flat by Proposition~\ref{prop:EP-flat}.  Put
\[
L(K):=\delta K+[\Theta_0,K]
\]
for the linearized homological operator.
We construct a holomorphic gauge
\[
H=I+\sum_{N\ge1}K^{(N)},
\qquad K^{(N)}\in\Mat_{e\times e}(I_D^N),
\]
by induction on the boundary order.  Suppose that \(H^{[N-1]}\) has already been constructed so that the
transformed connection is congruent to \(\Theta_0\) modulo \(I_D^N\), and let
\[
R^{(N)}:=(\delta H^{[N-1]})(H^{[N-1]})^{-1}
       +H^{[N-1]}\Theta(H^{[N-1]})^{-1}-\Theta_0
\]
be its order \(N\) defect.  Then
\(R^{(N)}\in\Mat_{e\times e}(I_D^N\fol_\sigma^{\log})\).  Flatness of both \(\Theta\) and \(\Theta_0\), together
with the induction hypothesis, gives the linearized Maurer--Cartan compatibility condition
\[
        \delta R^{(N)}+[\Theta_0,R^{(N)}]\equiv0
        \pmod {I_D^{N+1}}.
\]
Moreover the induction is performed inside the fixed-principal-part subspace, so the order \(N\) defect is
residue-free in the logarithmic meridional directions.
By Lemma~\ref{lem:EP-homological-equation}, the homological equation
\[
        L(K^{(N)})\equiv -R^{(N)}\pmod{I_D^{N+1}}
\]
has a solution with \(K^{(N)}\in\Mat_{e\times e}(I_D^N)\).  Taking
\[
H^{[N]}=(I+K^{(N)})H^{[N-1]}
\]
therefore improves the normalization by one boundary order.  The quadratic terms in \(K^{(N)}\) lie in
\(I_D^{2N}\subset I_D^{N+1}\) for \(N\ge1\), so no further order \(N\) contribution appears.

This produces a formal gauge.  The lower-bound clause in the effective non-resonance condition gives the
standard Poincar\'e majorant estimate for the off-diagonal homological inverses, while the diagonal primitives
are obtained by the usual holomorphic homotopy operators on polydiscs and preserve the \(I_D\)-order by
Cauchy estimates.  Hence the formal product converges on a smaller polydisc.  The limiting gauge
\(H\in\GL_e(\OO_U)\) transforms \(\Theta\) into a pure Euler--Poisson normal form
\[
        \widetilde\Theta=\sum_{i=1}^r\widetilde A_iX_i,
\]
with constant commuting residues; commutativity follows from flatness of the transformed principal part.
\end{proof}

\subsection{Uniqueness up to Casimir--centralizer gauge}\label{subsec:uniq-casimir}

Assume that the connection is already in pure Euler--Poisson normal form,
\[
        \Theta_{\mathrm{nf}}
        =
        \sum_{i=1}^r \widetilde A_iX_i,
        \qquad
        [\widetilde A_i,\widetilde A_j]=0.
\]
A holomorphic function $f\in\OO_U$ is a Casimir if $\delta f=0$.  A holomorphic
matrix is Casimir-valued if all its entries are Casimirs.

\begin{definition}\label{def:centralizer}
The effective centralizer of the normal form is
\[
        Z_{\mathrm{eff}}(\widetilde A_\bullet)
        :=
        \left\{
        M\in\GL(e,\C)
        \ ;\
        \sum_{i=1}^r [M,\widetilde A_i]\otimes [e_i]_\sigma=0
        \text{ in }\End(\C^e)\otimes\mathfrak t_{\sigma,p}
        \right\}.
\]
If $X_1,\ldots,X_r$ are independent modulo $I_D\fol_\sigma^{\log}$, this is the
ordinary simultaneous centralizer of the matrices $\widetilde A_i$.
\end{definition}

\begin{figure}[t]
\centering
\begin{tikzcd}[row sep=large, column sep=huge]
\Theta \arrow[r, shift left=1.1ex, "{H_1}"] \arrow[r, shift right=1.1ex, swap, "{H_2}"] &
\Theta_{\mathrm{nf}} \arrow[loop right, distance=2.2em, "{G:=H_2H_1^{-1}}"]
\end{tikzcd}
\caption{If two gauges send $\Theta$ to the same Euler--Poisson normal form, their ratio stabilizes the normal form.}
\label{fig:uniq-torsor}
\end{figure}

\begin{lemma}\label{lem:casimir-centralizer-rigidity}
Let $G\in\GL_e(\OO_U)$ be holomorphic and single-valued.  If
\[
        (\delta G)G^{-1}+G\Theta_{\mathrm{nf}}G^{-1}=\Theta_{\mathrm{nf}},
\]
then $\delta G=0$ and $G$ takes values in
$Z_{\mathrm{eff}}(\widetilde A_\bullet)$.
\end{lemma}

\begin{proof}
The stabilizer equation is equivalent to
\[
        \delta G=[\Theta_{\mathrm{nf}},G].
\]
Taking the logarithmic residue along the effective boundary generators gives
\[
        \sum_{i=1}^r [G,\widetilde A_i]\otimes [e_i]_\sigma=0
        \quad\text{in }\End(\C^e)\otimes\mathfrak t_{\sigma,p}.
\]
Thus the right-hand side of the stabilizer equation has zero effective
logarithmic residue.  The equation is now a residue-free logarithmic equation
for $G$.  Since $G$ is holomorphic and single-valued, the logarithmic primitive
terms which would appear from a non-zero meridional part are forbidden.  The
same residue-free primitive argument used in Lemma~\ref{lem:residue-free-EP-acyclicity}
therefore gives $\delta G=0$.  Substituting this back into the stabilizer
equation gives the displayed effective centralizer condition.
\end{proof}

\begin{theorem} \label{thm:uniq}
Work in an adapted SNC chart $U\ni p$ and assume that
\[
        \Theta_{\mathrm{nf}}
        =
        \sum_{i=1}^r \widetilde A_iX_i,
        \qquad
        [\widetilde A_i,\widetilde A_j]=0.
\]
Let $G\in\GL_e(\OO_U)$ satisfy
\[
        (\delta G)G^{-1}+G\Theta_{\mathrm{nf}}G^{-1}=\Theta_{\mathrm{nf}}.
\]
Then $G$ is Casimir-valued and takes values in
$Z_{\mathrm{eff}}(\widetilde A_\bullet)$.  Consequently, if $H_1$ and $H_2$ are
two holomorphic gauges sending the same $\Theta$ to the same normal form
$\Theta_{\mathrm{nf}}$, then $H_2H_1^{-1}$ is Casimir-valued and
$Z_{\mathrm{eff}}(\widetilde A_\bullet)$-valued.
\end{theorem}

\begin{proof}
Apply Lemma~\ref{lem:casimir-centralizer-rigidity} to $G$.  For two normalizing
gauges, apply the same statement to $G=H_2H_1^{-1}$.
\end{proof}

\subsection{A twisted leafwise groupoid near an SNC divisor  }
\label{subsec:twisted-leafwise}

Let $p\in D\cap X^\circ$. After shrinking an adapted polydisc $V\subset X^\circ$ around $p$ and reindexing the
local irreducible components of $D$ through $p$, we may assume there are holomorphic coordinates
$(z_1,\dots,z_n)$ on $V$ such that
\begin{equation}\label{eq:SNC-local}
D\cap V=\{z_1 z_2\cdots z_r=0\},
\qquad
p\in \{z_1=\cdots=z_r=0\}.
\end{equation}
Write $V^*:=V\setminus D$.

\subsubsection*{1. The oriented blow-up and canonical angular directions}

Let $\widetilde V:=\Blor_{D\cap V}(V)$ be the real oriented blow-up along $D\cap V$ \cite{Sabbah13}.
On a collar neighborhood of the boundary we may use polar coordinates
\begin{equation}\label{eq:polar}
z_i=\rho_i e^{i\theta_i},
\qquad
\rho_i\in[0,\varepsilon),
\ \theta_i\in\mathbb R/2\pi\mathbb Z,
\qquad
i=1,\dots,r,
\end{equation}
together with $z_{r+1},\dots,z_n$ as interior coordinates. Then $\widetilde V$ is a real manifold with corners,
$\mathrm{int}(\widetilde V)\simeq V^*$, and the boundary stratification is indexed by subsets
$I\subset\{1,\dots,r\}$ via
\[
\partial_I\widetilde V=\{\rho_i=0\ (i\in I),\ \rho_j>0\ (j\notin I)\}.
\]
At the deepest corner stratum $\partial_{\{1,\dots,r\}}\widetilde V$ over $p$, the angular variables
$\theta_1,\dots,\theta_r$ parametrize a torus $(S^1)^r$.
For later use, keep in mind the standard local model (drawn below for $r=2$):
\[
\widetilde V \simeq [0,\varepsilon)^r\times (S^1)^r \times \mathbf D^{2(n-r)}.
\]
In particular, when $r=2$ the base corner variables $(\rho_1,\rho_2)\in[0,\varepsilon)^2$ record approach to
$D_1\cup D_2$, while $(\theta_1,\theta_2)$ are the canonical angular directions.

\begin{figure}[t]
\centering
\begin{tikzpicture}[x=1cm,y=1cm,>=Stealth,font=\small,line join=round,line cap=round]
 
  \def\W{4.4}   
  \def\H{3.0}    
  \def\DX{1.45}   
  \def\DY{0.75}

  \coordinate (O) at (0,0);
  \coordinate (A) at (\W,0);
  \coordinate (B) at (\W,\H);
  \coordinate (C) at (0,\H);

  \coordinate (O2) at (\DX,\DY);
  \coordinate (A2) at (\W+\DX,\DY);
  \coordinate (B2) at (\W+\DX,\H+\DY);
  \coordinate (C2) at (\DX,\H+\DY);

  \fill[gray!8]  (O) -- (A) -- (B) -- (C) -- cycle;          
  \fill[gray!10] (O) -- (A) -- (A2) -- (O2) -- cycle;        
  \fill[gray!12] (O) -- (C) -- (C2) -- (O2) -- cycle;        
  \fill[gray!6]  (O2) -- (A2) -- (B2) -- (C2) -- cycle;

  \draw[thick] (O) -- (A) -- (B) -- (C) -- cycle;
  \draw[thick] (O2) -- (A2) -- (B2) -- (C2) -- cycle;
  \draw[thick] (O) -- (O2);
  \draw[thick] (A) -- (A2);
  \draw[thick] (B) -- (B2);
  \draw[thick] (C) -- (C2);

  \draw[ultra thick] (O) -- (A);  
  \draw[ultra thick] (O) -- (C);

  \draw[->] (O) -- (\W+0.75,0) node[below] {$\rho_1$};
  \draw[->] (O) -- (0,\H+0.75) node[left] {$\rho_2$};
  \draw[->] (O) -- (\DX+0.35,\DY+0.20) node[above left] {$\mathbf D^{2(n-2)}$};

  \node at (0.62*\W,0.68*\H) {$\mathrm{int}(\widetilde V)\simeq V^*$};

  \node[rotate=90] at (-0.42,0.60*\H) {$\partial_{\{1\}}\widetilde V\ (\rho_1=0)$};
  \node at (0.58*\W,-0.40) {$\partial_{\{2\}}\widetilde V\ (\rho_2=0)$};

  \fill (O) circle (2pt);
  \node[anchor=north east] at (0,0) {$\partial_{\{1,2\}}\widetilde V$};

  \coordinate (S1) at (\W+\DX+2.25,\H+\DY-0.55);
  \coordinate (S2) at (\W+\DX+2.25,\DY+1.00);
  \coordinate (T0) at (\W+\DX+1.95,\DY-0.15);

  \draw[densely dashed,->] ($(C)!0.62!(C2)$) -- (S1);
  \draw[densely dashed,->] ($(A)!0.55!(A2)$) -- (S2);
  \draw[densely dashed,->] ($(O)!0.55!(O2)$) -- (T0);

  \draw (S1) circle (0.33);
  \draw[->] ($(S1)+(0,0.33)$) arc (90:420:0.33);
  \node[anchor=west] at ($(S1)+(0.55,0)$) {$S^1_{\theta_1}$};

  \draw (S2) circle (0.33);
  \draw[->] ($(S2)+(0,0.33)$) arc (90:420:0.33);
  \node[anchor=west] at ($(S2)+(0.55,0)$) {$S^1_{\theta_2}$};

  \begin{scope}[shift={(T0)}]
    \def\Rt{0.62}       
    \def\tube{2.6mm}   

    \begin{scope}[yscale=0.32]  
      \draw[very thick,double=white,double distance=\tube,line cap=round,line join=round]
        (0:\Rt) arc (0:180:\Rt);
      \draw[very thick,double=white,double distance=\tube,line cap=round,line join=round]
        (180:\Rt) arc (180:360:\Rt);

      \draw[->] (90:\Rt) arc (90:420:\Rt);  
    \end{scope}
  \end{scope}
  \node[anchor=west] at ($(T0)+(1.55,0)$) {$(S^1)^2_{\theta_1,\theta_2}$};

\end{tikzpicture}
\caption{Schematic local model for the oriented blow-up when $r=2$: the corner variables $(\rho_1,\rho_2)\in[0,\varepsilon)^2$ encode approach to $D_1\cup D_2$; each boundary face carries the canonical angular circle; over the corner stratum the angular fiber is $(S^1)^2$. The full local chart is $[0,\varepsilon)^2\times (S^1)^2\times \mathbf D^{2(n-2)}$.}
\label{fig:corners}
\end{figure}

\begin{remark}\label{rem:not-generated}
The usual leafwise fundamental groupoid $\Mon(\fol)$ is not defined as generated by meridians inside leaves:
its arrows are leafwise paths modulo leafwise homotopy. In general, a geometric meridian around $D_i$ is an ambient
loop in $V^*$ and need not be leafwise. The twisted/tangential construction below adds a \emph{canonical} meridional
isotropy at the boundary (tangential basepoints), so residue and winding data have a canonical boundary home without
choosing leafwise meridian loops.
\end{remark}

\subsubsection*{ Leafwise fundamental groupoids and logarithmic lifting}

Let $\fol:=\fol_\sigma|_{V^*}$ be the Hamiltonian foliation on $V^*$ (viewed as a real foliation with complex leaves),
and let $\Mon(\fol)|_{V^*}\rightrightarrows V^*$ denote its leafwise fundamental groupoid: objects are points of $V^*$,
and arrows are leafwise $C^1$ paths modulo leafwise homotopy relative endpoints.
In the Poisson normal-form charts used later, the foliation is generated by holomorphic logarithmic vector fields
along $D\cap V$ (sections of $T_V(-\log(D\cap V))$). Under this hypothesis, the foliation extends canonically to the
oriented blow-up:

\subsubsection*{ The twisted groupoid by tangential completion (2-pushout)}

Set $U:=\mathrm{int}(\widetilde V)\simeq V^*$. Choose a collar neighborhood
\[
C\subset \widetilde V
\qquad\text{with}\qquad
C\simeq [0,\varepsilon)\times \partial\widetilde V,
\]
so that $U\cap C$ is an open neighborhood of $\partial\widetilde V$ inside $U$.
In an adapted SNC chart, a logarithmic holomorphic vector field can be written as
$X=\sum_{i\le r} z_i b_i(z)\partial_{z_i}+\sum_{j>r} b_j(z)\partial_{z_j}$.
On the oriented blow-up with $z_i=\rho_i e^{i\theta_i}$, the identity
$z_i\partial_{z_i}=\tfrac12(\rho_i\partial_{\rho_i}-i\partial_{\theta_i})$
shows that $\Re(X)$ and $\Im(X)$ lift to smooth real vector fields on $\widetilde V$
tangent to each boundary face $\{\rho_i=0\}$.
Therefore, whenever $\fol|_{U\cap C}$ is generated by logarithmic fields (as in our normal-form charts),
it extends canonically to a foliation $\widetilde\fol$ on the collar $C$.

Assume henceforth that we have fixed such an extension $\widetilde\fol$ on $C$. We adjoin meridional isotropy only over the boundary of the collar, not over all of $C$. Set
$$
C_\partial:=\partial\widetilde V
$$
inside the collar, and define the boundary meridional groupoid
$$
\mathcal T_\partial:=C_\partial\times \ZZ^r\rightrightarrows C_\partial.
$$
The $i$-th standard generator corresponds to the positive $2\pi$-rotation $\theta_i\mapsto\theta_i+2\pi$ of the angular variable along the boundary. First form the boundary-enlarged collar groupoid
$$
\mathcal C^{\mathrm{tw}}
:=
\Mon(\widetilde\fol)|_{C}
\sqcup_{\Mon(\widetilde\fol)|_{C_\partial}}
\bigl(\Mon(\widetilde\fol)|_{C_\partial}\times_{C_\partial} \mathcal T_\partial\bigr).
$$
On the overlap $U\cap C$ we have $\Mon(\widetilde\fol)|_{U\cap C}=\Mon(\fol)|_{U\cap C}$, and the natural functor to $\mathcal C^{\mathrm{tw}}$ is the usual inclusion with zero meridional label.

\begin{definition}[Boundary-twisted leafwise groupoid via 2-pushout]\label{def:twisted-pushout}
Define the boundary-twisted leafwise groupoid $\Mon^{\mathrm{tw}}(\fol;D)\rightrightarrows \widetilde V$ as the 2-pushout in topological groupoids
$$
\Mon^{\mathrm{tw}}(\fol;D)
:=
\Mon(\fol)|_{U}
\sqcup_{\Mon(\fol)|_{U\cap C}}
\mathcal C^{\mathrm{tw}}.
$$
Thus the interior restriction is unchanged:
$$
\Mon^{\mathrm{tw}}(\fol;D)|_U=\Mon(\fol)|_U.
$$
\end{definition}

\begin{figure}[t]
\centering
\begin{tikzcd}[row sep=huge, column sep=huge]
\Mon(\fol)|_{U\cap C} \arrow[r, "{\text{zero label}}"] \arrow[d] &
\mathcal C^{\mathrm{tw}} \arrow[d] \\
\Mon(\fol)|_{U} \arrow[r] &
\Mon^{\mathrm{tw}}(\fol;D)
\end{tikzcd}
\caption{The boundary-twisted leafwise groupoid is obtained by adjoining canonical tangential meridians only along the oriented boundary and gluing this boundary-enlarged collar groupoid to the ordinary interior leafwise groupoid.}
\label{fig:pushout}
\end{figure}

\begin{remark}\label{rem:canonical-meridians}
Let $\tilde x\in\partial_{\{1,\dots,r\}}\widetilde V$ lie over $p$. For each $i=1,\dots,r$ the loop $t_i$ that rotates $\theta_i$ once defines an isotropy element at $\tilde x$ coming from $\mathcal T_\partial$, and hence in $\Mon^{\mathrm{tw}}(\fol;D)$. These isotropy elements commute and generate a canonical subgroup
$$
\ZZ^r\subset \mathrm{Iso}_{\Mon^{\mathrm{tw}}(\fol;D)}(\tilde x),
$$
independent of any choice of small meridians inside punctured leaves.
\end{remark}

\begin{proposition} \label{prop:pushout}
The groupoid $\Mon^{\mathrm{tw}}(\fol;D)$ is generated by the interior leafwise arrows, the collar leafwise arrows, and the canonical commuting boundary meridians, with the relations already present in the two leafwise groupoids, the overlap identifications on $U\cap C$, and the commutativity of the meridional group $\ZZ^r$.
\end{proposition}

\begin{proof}
This is the universal property of the two pushouts above. The first pushout enlarges the collar groupoid only along $C_\partial$ by adjoining the commuting tangential meridians. The second glues this boundary-enlarged collar groupoid to the ordinary interior leafwise groupoid along their common restriction to $U\cap C$.
\end{proof}

\subsubsection*{Meridional isotropy and the residue character from the PPD normal form}

Assume $\nabla^{\mathrm{Pois}}$ is Poisson-flat, has fixed Euler--Poisson principal part, and is effectively non-resonant at $p$ so that Theorem~\ref{thm:PPD-local} provides,
after shrinking $V$, an Euler--Poisson gauge with constant commuting residues
$\widetilde A_1,\dots,\widetilde A_r\in\mathrm{Mat}_{e\times e}(\CC)$ such that
\begin{equation}\label{eq:PPD-again}
\widetilde\Theta
=\sum_{i=1}^r \widetilde A_i\,\sigma^\# \Big(\frac{dz_i}{z_i}\Big),
\qquad
[\widetilde A_i,\widetilde A_j]=0.
\end{equation}

Let $\rho:\Mon(\fol)|_{U}\to \GL_e(\CC)$ be the leafwise monodromy representation determined by $\nabla^{\mathrm{Pois}}$
on $U\simeq V^*$.

\begin{proposition}[Meridional character]\label{prop:meridional-character}
The residue tuple defines a canonical character of the boundary meridional subgroup by
\begin{equation}\label{eq:char}
\ZZ^r\longrightarrow \GL_e(\CC),
\qquad
m=(m_1,\dots,m_r)\longmapsto
\exp\Big(-2\pi i\sum_{i=1}^r m_i\,\widetilde A_i\Big).
\end{equation}
Equivalently, the $i$-th canonical boundary meridian acts by $\exp(-2\pi i\,\widetilde A_i)$. Together with the interior leafwise monodromy, this character gives a representation of $\Mon^{\mathrm{tw}}(\fol;D)$ whenever, for every boundary leafwise arrow $\gamma$ for which both sides are defined, the transported residue tuple is preserved, or equivalently
\[
\rho(\gamma)\,\exp(-2\pi i\,\widetilde A_i)\,\rho(\gamma)^{-1}=\exp(-2\pi i\,\widetilde A_i)
\qquad(1\le i\le r).
\]
\end{proposition}

\begin{proof}
The matrices $\widetilde A_i$ commute, hence the formula \eqref{eq:char} is a group homomorphism from $\ZZ^r$ to $\GL_e(\CC)$. The sign is dictated by the horizontal transport convention $\delta Y=-\Theta Y$. By the pushout description, a representation of the boundary-twisted groupoid is exactly a pair consisting of the interior leafwise representation and a representation of the boundary meridional isotropy whose restrictions agree on the overlap and satisfy the collar relations. These relations amount to invariance of the meridional automorphisms under boundary leafwise transport, which is precisely the displayed centralizer condition.
\end{proof}

\begin{remark} \label{rem:residues-control}
Equation \eqref{eq:char} should be read as follows: the residue matrices constrain exactly the restriction of the
(monodromy) representation to the canonical tangential meridians. Monodromy along other leafwise loops inside punctured
leaves depends on the global topology/dynamics of the leaves and is not determined by residues alone.
\end{remark}

\section{Rank two meromorphic Poisson modules and Poisson triples}\label{sec:rank2-triples}

Throughout this section $X$ denotes a smooth complex projective variety and
$\sigma\in H^0(X,\wedge^2T_X)$ a holomorphic Poisson bivector. We write
$\sigma^\#: \Omega_X^1\to T_X$ for the anchor. We use the Poisson, or Lichnerowicz,
differential
\[
        \delta=[\sigma,-]
\]
on polyvector fields, with the Schouten convention
\[
\delta(f)=\sigma^\#(df)=:X_f\in H^0(X,T_X),
\qquad
\delta(v)=[\sigma,v]=-L_v(\sigma)\in H^0(X,\wedge^2T_X)
\]
for holomorphic functions $f$ and vector fields $v$.

\subsection{Hamiltonians, Casimirs, and Poisson vector fields}

A holomorphic vector field of the form $X_f=\delta(f)$ is called \emph{Hamiltonian}.
A holomorphic function $f$ is a \emph{Casimir} if $X_f=0$, i.e.\ $\delta(f)=0$.
A holomorphic vector field $v$ is \emph{Poisson} if $\delta(v)=0$, equivalently $L_v(\sigma)=0$.
Thus $H_\sigma^0(X)=\ker(\delta:\OO_X\to T_X)$ is the space of Casimirs and
$$H_\sigma^1(X)=\ker(\delta:T_X\to\wedge^2T_X)/\im(\delta:\OO_X\to T_X)$$ measures the obstruction for a Poisson
vector field to be Hamiltonian.

\subsection{Meromorphic Poisson modules and trace-free rank two connections}

We use the following working definition (equivalent to the standard Lie algebroid formulation on the smooth locus).

\begin{definition}[Meromorphic Poisson module]
A \emph{meromorphic Poisson module} on $(X,\sigma)$ is a holomorphic vector bundle $E\to X$ together with a
$\mathbb C$-linear operator
$$
\nabla^{\mathrm{Pois}}:\ E\longrightarrow T_X(*D)\otimes E
$$
(for some divisor $D$, and with poles along $D$) such that for any holomorphic function $f$ and local section $s$,
$$
\nabla^{\mathrm{Pois}}(fs)=\delta(f)\otimes s+f\,\nabla^{\mathrm{Pois}}(s),
$$
and whose curvature (in the Poisson Lie algebroid sense) vanishes on $X\setminus D$.
When $D=\emptyset$ we simply say \emph{Poisson module}.
\end{definition}

The case $rk(E)=2$ admits a very concrete description after fixing a trivialization.

\begin{definition} 
Assume $E\simeq \OO_X^{\oplus 2}$ on an open set $U\subset X$ with frame $e=(e_1,e_2)$.
A \emph{trace-free} Poisson connection on $E|_U$ is given by a matrix
$$
\Theta=\begin{pmatrix} \frac12 v & w\\ u & -\frac12 v\end{pmatrix},
\qquad
u,v,w\in \mathfrak X_\sigma(U),
$$
acting by
$$
\nabla^{\mathrm{Pois}}(e)=\Theta\cdot e,
\qquad
\nabla^{\mathrm{Pois}}(z_1,z_2)=\big(\delta(z_1),\delta(z_2)\big)+\Theta\cdot(z_1,z_2).
$$
Here $\mathfrak X_\sigma(U)$ denotes the sheaf of (meromorphic) Hamiltonian vector fields on $U$
(or, more invariantly, the image of $\sigma^\#$).
\end{definition}

\subsection{Poisson triples and the Maurer--Cartan system}

The following is the rank two form of the Lie--algebroid description of Poisson modules.  The normalization
chosen here is adapted to the Schouten convention \(\delta(v)=[\sigma,v]=-L_v\sigma\) fixed above.

Let \(E\) be a Poisson module of rank \(2\) on a Poisson projective variety \((X,\sigma)\).  Locally on a
Zariski open set where \(E\) trivializes, the trace-free connection matrix may be written as
\[
\Theta=
\begin{pmatrix}
\frac12 v & w\\
u & -\frac12 v
\end{pmatrix},
\qquad
u,v,w\in\mathfrak X_\sigma(U).
\]
The Poisson-flatness equation is
\[
        \delta\Theta+\Theta\wedge\Theta=0.
\]
A direct computation in \(\mathfrak{sl}_2\) gives
\[
\delta\Theta+\Theta\wedge\Theta=
\begin{pmatrix}
\frac12\delta(v)+u\wedge w & \delta(w)+v\wedge w\\
\delta(u)+u\wedge v & -\frac12\delta(v)-u\wedge w
\end{pmatrix}.
\]
Thus flatness is equivalent to the system
\begin{equation}\label{eq:MC-rank2}
\delta(u)=-u\wedge v,
\qquad
\delta(v)=-2u\wedge w,
\qquad
\delta(w)=-v\wedge w.
\end{equation}
Equivalently, in Lie-derivative notation,
\begin{equation}\label{eq:MC-rank2-Lie}
L_u(\sigma)=u\wedge v,
\qquad
L_v(\sigma)=2u\wedge w,
\qquad
L_w(\sigma)=v\wedge w.
\end{equation}
Conversely, a triple \((u,v,w)\) satisfying \eqref{eq:MC-rank2} defines a flat trace-free Poisson connection
on \(\OO_U^{\oplus2}\), see \cite{Polishchuk97,LGV13,Cor20}.

\begin{remark}
Once one works in a gauge where the Poisson connection is logarithmic, or Euler--Poisson in the effectively non-resonant regime, the matrix \(\Theta\) above becomes a logarithmic \(\mathfrak{sl}_2\)-system along
Hamiltonian leaves.  The preceding local linearisation theorem applies to those rank two systems satisfying
the fixed-principal-part and effective non-resonance hypotheses, and the consequences for
monodromy/holonomy groupoids are then read off from the residue tuple.
\end{remark}

\subsection{A coordinate criterion for Poisson vector fields in affine space}

We record a basic criterion that will be useful in explicit families of quadratic Poisson structures.
Let $(z_0,\dots,z_n)$ be coordinates on $\mathbb C^{n+1}$ and write
$$
\sigma=\sum_{i<j}\sigma_{ij}(z)\,\partial_i\wedge\partial_j,
\qquad
\{z_i,z_j\}:=\sigma_{ij}.
$$
Consider a holomorphic vector field $v=\sum_{k=0}^n v^k\partial_k$ such that each $v^k$ depends only on $z_k$.

\begin{proposition}\label{prop:coord-criterion}
Under the hypothesis above, $v$ is a Poisson vector field ($L_v\sigma=0$) if and only if
\begin{equation}\label{eq:coord-criterion}
\sum_{k=0}^n v^k\,\frac{\partial}{\partial z_k}\{z_i,z_j\}
=
\left(\frac{\partial v^i}{\partial z_i}+\frac{\partial v^j}{\partial z_j}\right)\{z_i,z_j\}
\qquad\text{for all } i<j.
\end{equation}
\end{proposition}

\begin{proof}
A standard Lie-derivative computation gives
$$
L_v(\sigma)=\sum_{i<j}\left(
\sum_{k=0}^n v^k\frac{\partial\sigma_{ij}}{\partial z_k}
-\sigma_{ij}\frac{\partial v^i}{\partial z_i}
-\sigma_{ij}\frac{\partial v^j}{\partial z_j}
\right)\partial_i\wedge\partial_j,
$$
where the hypothesis $v^k=v^k(z_k)$ eliminates mixed partials.
Thus $L_v(\sigma)=0$ is equivalent to the coefficient identity in \eqref{eq:coord-criterion}.
\end{proof}

\begin{remark} \label{rem:diag-quad}
 If $\{z_i,z_j\}=c_{ij}z_iz_j$ with constants $c_{ij}\in\mathbb C$ (a diagonal quadratic Poisson structure) and
$v^k=a_k z_k$, then \eqref{eq:coord-criterion} holds for all $i<j$, hence $L_v(\sigma)=0$.
 For the radial vector field $\mathrm R=\sum_k z_k\partial_k$, the condition $L_{\mathrm R}\sigma=0$ is
equivalent to $\mathrm R(\sigma_{ij})=2\sigma_{ij}$ for all $i<j$, i.e.\ $\sigma$ is homogeneous of degree $2$.
\end{remark}

\subsection{Example: a local affine model associated with the diagonal quadratic component}

We now specialize to the affine cone model associated with the diagonal quadratic component $L(1,1,1,1)$ on $\mathbb P^3$ \cite{Pym18}. We use coordinates $(z_0,z_1,z_2,z_3)$ on $\mathbb C^4$ and brackets $\{z_i,z_j\}=c_{ij}z_iz_j$. Constant vector fields below are vector fields on the affine cone, or on a chosen affine chart; they are not meant to define global vector fields on $\mathbb P^3$ without checking projectability.
In the component usually denoted $L(1,1,1,1)$ one may parametrize the constants by
$a_0+a_1+a_2+a_3=0$ and
\begin{align*}
c_{01}&=a_3-a_2, & c_{02}&=a_1-a_3, & c_{03}&=a_2-a_1,\\
c_{12}&=a_3-a_0, & c_{13}&=a_0-a_2, & c_{23}&=a_1-a_0.
\end{align*}
By Remark~\ref{rem:diag-quad}, every diagonal linear vector field
$u=\sum_{k=0}^3 b_k z_k\partial_k$ preserves $\sigma$.

The next statement describes when a \emph{constant} vector field $u$ satisfies $L_u(\sigma)=u\wedge v$ for a diagonal
linear field $v$; this is the basic input for constructing explicit Poisson triples.

\begin{proposition}\label{prop:Lu=uw}
Let $u=\sum_{k=0}^3 c_k\,\partial_k$ be constant and $v=\sum_{k=0}^3 b_k z_k\partial_k$ diagonal linear.
Then $L_u(\sigma)=u\wedge v$ if and only if for every $i<j$ one has
$$
c_i(c_{ij}-b_j)=0,
\qquad
c_j(c_{ij}+b_i)=0.
$$
\end{proposition}
\begin{proof}
Write
$$
\sigma=\sum_{0\le i<j\le 3} c_{ij}\,z_i z_j\,\partial_i\wedge\partial_j,
\qquad
u=\sum_{k=0}^3 c_k\,\partial_k,
\qquad
v=\sum_{k=0}^3 b_k z_k\,\partial_k.
$$
Since $u$ is constant, we have $[u,\partial_i]=0$ for all $i$, hence
$L_u(\partial_i\wedge\partial_j)=0$. Therefore
\begin{align*}
L_u(\sigma)
&=\sum_{i<j} c_{ij}\,L_u\bigl(z_i z_j\bigr)\,\partial_i\wedge\partial_j \\
&=\sum_{i<j} c_{ij}\,\bigl(u(z_i)z_j+z_i u(z_j)\bigr)\,\partial_i\wedge\partial_j \\
&=\sum_{i<j} c_{ij}\,\bigl(c_i z_j+c_j z_i\bigr)\,\partial_i\wedge\partial_j,
\end{align*}
because $u(z_i)=c_i$.
On the other hand,
\begin{align*}
u\wedge v
&=\sum_{i<j}\bigl(u^i v^j-u^j v^i\bigr)\,\partial_i\wedge\partial_j \\
&=\sum_{i<j}\bigl(c_i\,(b_j z_j)-c_j\,(b_i z_i)\bigr)\,\partial_i\wedge\partial_j \\
&=\sum_{i<j}\bigl(c_i b_j z_j-c_j b_i z_i\bigr)\,\partial_i\wedge\partial_j,
\end{align*}
since $u^i=c_i$ and $v^j=b_j z_j$.
Thus $L_u(\sigma)=u\wedge v$ holds if and only if for every $i<j$ the coefficients of $z_j$ and of $z_i$
in the $(i,j)$-component coincide. Equivalently, for each $i<j$ one must have
$$
c_{ij}\,c_i = c_i\,b_j,
\qquad
c_{ij}\,c_j = -\,c_j\,b_i.
$$
Rewriting, this is precisely
$$
c_i(c_{ij}-b_j)=0,
\qquad
c_j(c_{ij}+b_i)=0,
$$
for all $i<j$, as claimed.
\end{proof}

\begin{corollary}\label{cor:u=partial0}
Taking $u=\partial_0$, the condition of Proposition~\ref{prop:Lu=uw} forces
$b_1=c_{01}$, $b_2=c_{02}$, $b_3=c_{03}$, while $b_0$ remains free. Thus one may take
$$
v=z_0\partial_0+c_{01}z_1\partial_1+c_{02}z_2\partial_2+c_{03}z_3\partial_3,
$$
and then $L_u(\sigma)=u\wedge v$.
\end{corollary}

\begin{remark}
Once explicit pairs $(u,v)$ with $L_u(\sigma)=u\wedge v$ are produced, one seeks $w$ so that the full
Maurer--Cartan system \eqref{eq:MC-rank2} holds; in rank two this is exactly the flatness condition for the
trace-free Poisson connection matrix \[
\Theta \;=\;
\begin{pmatrix}
v & w \\
u & -v
\end{pmatrix}.
\]
This is the point where our non-resonant Poincar\'e--Dulac linearization becomes effective in later sections:
after logarithmic normalization, the local groupoid/meridional monodromy is read off from residues of $\Theta$.
\end{remark}

  Let $V\subset\mathbb{C}^4$ be a sufficiently small polydisc with holomorphic coordinates $(z_0,z_1,z_2,z_3)$, and set
\begin{equation}\label{eq:D-U}
D:=\{z_1z_2z_3=0\}\subset V,
\qquad
U:=V\setminus D=\{z_1\neq 0,\ z_2\neq 0,\ z_3\neq 0\}.
\end{equation}
Write $\partial_k:=\partial/\partial z_k$. Fix a log-canonical Poisson bivector with constant coefficients $c_{ij}\in\mathbb{C}$,
\begin{equation}\label{eq:sigma}
\sigma=\sum_{0\le i<j\le 3} c_{ij}\,z_i z_j\,\partial_i\wedge\partial_j.
\end{equation}
On $U$ the logarithmic $1$-forms $dz_i/z_i$ exist for $i=1,2,3$, and we define the corresponding Hamiltonian logarithmic vector fields
\begin{equation}\label{eq:Xi-def}
X_i:=\sigma^\# \Big(\frac{dz_i}{z_i}\Big)\in H^0(U,T_U),\qquad i=1,2,3.
\end{equation}
To compute them explicitly we use the contraction rule that for any $1$-form $\alpha$ and any decomposable bivector $f\,\partial_a\wedge\partial_b$ one has
\begin{equation}\label{eq:contraction}
\iota_\alpha \bigl(f\,\partial_a\wedge\partial_b\bigr)
=f\bigl(\alpha(\partial_a)\,\partial_b-\alpha(\partial_b)\,\partial_a\bigr).
\end{equation}
Taking $\alpha=dz_i/z_i$ gives $\alpha(\partial_k)=\delta_{ik}/z_i$, and summing the contributions of all terms in \eqref{eq:sigma} that contain the index $i$ yields the closed formula
\begin{equation}\label{eq:Xi-closed}
X_i=\sum_{j>i} c_{ij}\,z_j\partial_j-\sum_{k<i} c_{k i}\,z_k\partial_k,\qquad i=1,2,3.
\end{equation}
In particular,
\begin{align}
X_1 &= -c_{01} z_0\partial_0 + c_{12} z_2\partial_2 + c_{13} z_3\partial_3,\label{eq:X1}\\
X_2 &= -c_{02} z_0\partial_0 - c_{12} z_1\partial_1 + c_{23} z_3\partial_3,\label{eq:X2}\\
X_3 &= -c_{03} z_0\partial_0 - c_{13} z_1\partial_1 - c_{23} z_2\partial_2.\label{eq:X3}
\end{align}
Let $\mathcal F$ be the real foliation on $U$ generated by $\Re(X_i)$ and $\Im(X_i)$ for $i=1,2,3$; equivalently, $\mathcal F$ is determined by the complex distribution spanned by $X_1,X_2,X_3$. To organise the boundary geometry and the canonical meridians, pass to the real oriented blow-up $\widetilde V:=\mathrm{Bl}^{\mathrm{or}}_{D}(V)$ with blow-down map $\pi:\widetilde V\to V$. On a collar neighbourhood of the boundary we write
\begin{equation}\label{eq:polar-explicit}
z_i=\rho_i e^{i\theta_i},
\qquad
\rho_i\in[0,\varepsilon),\ \theta_i\in\mathbb{R}/2\pi\mathbb{Z},
\qquad i=1,2,3,
\end{equation}
keeping $z_0$ as an interior coordinate, so that locally
\begin{equation}\label{eq:local-product}
\widetilde V \cong \mathbf D^2_{z_0}\times[0,\varepsilon)^3_{(\rho_1,\rho_2,\rho_3)}\times (S^1)^3_{(\theta_1,\theta_2,\theta_3)}.
\end{equation}
Fix a point $p\in V$ with $z_1(p)=z_2(p)=z_3(p)=0$ (hence $p\in D$). The fibre of $\pi$ over $p$ is the deep-corner torus
\begin{equation}\label{eq:corner-fiber-explicit}
\pi^{-1}(p)=\{z_0=z_0(p),\ \rho_1=\rho_2=\rho_3=0,\ \theta_1,\theta_2,\theta_3\in\mathbb{R}/2\pi\mathbb{Z}\}\ \cong\ (S^1)^3,
\end{equation}
with angular coordinates $(\theta_1,\theta_2,\theta_3)$. The basic identity of smooth differential operators
\begin{equation}\label{eq:operator}
z_i\partial_{z_i}=\tfrac12\bigl(\rho_i\partial_{\rho_i}-i\partial_{\theta_i}\bigr),\qquad i=1,2,3,
\end{equation}
implies that the real and imaginary parts of any holomorphic logarithmic vector field lift smoothly to $\widetilde V$ and remain tangent to every boundary face. In particular, $\mathcal F$ extends canonically to a lifted foliation $\widetilde{\mathcal F}$ on any chosen collar
\begin{equation}\label{eq:collar}
C\subset\widetilde V,
\qquad
C\cong[0,\varepsilon)\times\partial\widetilde V.
\end{equation}
Although each $X_i$ is defined using $dz_i/z_i$, the explicit formula \eqref{eq:Xi-closed} involves diagonal terms $z_j\partial_{z_j}$ with $j\neq i$, and \eqref{eq:operator} shows that along the boundary face $\rho_j=0$ these contribute a surviving angular component $-\tfrac12\partial_{\theta_j}$; hence the angular directions contained in $T\widetilde{\mathcal F}$ along the corner torus \eqref{eq:corner-fiber-explicit} depend only on which of $c_{12},c_{13},c_{23}$ are nonzero. Now let $\Mon(\mathcal F)|_U\rightrightarrows U$ be the leafwise fundamental groupoid of $\mathcal F$ on $U$ (objects are points of $U$ and arrows are leafwise $C^1$ paths modulo leafwise homotopy relative endpoints), and let $\Mon(\widetilde{\mathcal F})|_C\rightrightarrows C$ be the analogous leafwise fundamental groupoid for the lifted foliation on the collar. We adjoin meridional isotropy only along the oriented boundary. Let
\begin{equation}\label{eq:meridional-T}
T_\partial:=\partial\widetilde V\times\mathbb{Z}^3\rightrightarrows \partial\widetilde V,
\end{equation}
whose arrows are pairs $(x,m)$ with $x\in\partial\widetilde V$ and $m\in\mathbb{Z}^3$, with source and target both equal to $x$. The standard basis of $\mathbb{Z}^3$ corresponds to one positive $2\pi$-rotation in each angular variable $\theta_1,\theta_2,\theta_3$. We first form the boundary-enlarged collar groupoid
\begin{equation}\label{eq:product-groupoid}
\mathcal C^{\mathrm{tw}}
:=
\Mon(\widetilde{\mathcal F})|_C
\sqcup_{\Mon(\widetilde{\mathcal F})|_{\partial\widetilde V}}
\bigl(\Mon(\widetilde{\mathcal F})|_{\partial\widetilde V}\times_{\partial\widetilde V}T_\partial\bigr).
\end{equation}
On the overlap $U\cap C$ there is a canonical functor
\begin{equation}\label{eq:ff-overlap}
\Mon(\widetilde{\mathcal F})|_{U\cap C}\longrightarrow \mathcal C^{\mathrm{tw}},
\qquad
\gamma\longmapsto\gamma \text{ with zero meridional label},
\end{equation}
and we identify $\Mon(\widetilde{\mathcal F})|_{U\cap C}\cong \Mon(\mathcal F)|_{U\cap C}$ because the foliations agree on $U\cap C$. The tangentially twisted leafwise groupoid is then defined as the pushout
\begin{equation}\label{eq:pushout}
\Mon^{\mathrm{tw}}(\mathcal F;D)
:=
\Mon(\mathcal F)|_U\ \sqcup_{\Mon(\widetilde{\mathcal F})|_{U\cap C}}\ \mathcal C^{\mathrm{tw}}.
\end{equation}
Thus no extra $\mathbb{Z}^3$-isotropy is introduced at interior points; the added meridional labels live on the oriented boundary.
Inside the boundary-enlarged collar groupoid there is the canonical embedding of pure meridians,
\begin{equation}\label{eq:meridional-embed}
T_\partial\hookrightarrow \mathcal C^{\mathrm{tw}},
\qquad
(x,m)\longmapsto (\mathrm{id}_x,m),
\end{equation}
and pushout functoriality turns it into a canonical meridional subgroupoid of $\Mon^{\mathrm{tw}}(\mathcal F;D)$. Restricting to the corner torus \eqref{eq:corner-fiber-explicit} yields the constant bundle groupoid
\begin{equation}\label{eq:corner-bundle-groupoid-explicit}
\pi^{-1}(p)\times\mathbb{Z}^3\rightrightarrows \pi^{-1}(p),
\end{equation}
so the meridional isotropy at each object of $\pi^{-1}(p)\cong(S^1)^3$ is canonically $\mathbb{Z}^3$. The full restriction of $\Mon^{\mathrm{tw}}(\mathcal F;D)$ to $\pi^{-1}(p)$ also contains the leafwise fundamental group contributed by $\widetilde{\mathcal F}$ along the corner: the angular directions present along $\pi^{-1}(p)$ integrate to a subtorus $(S^1)^k\subset (S^1)^3$ with fundamental group $\mathbb{Z}^k$, and the Morita type is
\begin{equation}\label{eq:morita-general-explicit}
\Mon^{\mathrm{tw}}(\mathcal F;D)\big|_{\pi^{-1}(p)}\ \simeq_{\mathrm{Morita}}\ \mathbb{Z}^k\times\mathbb{Z}^3.
\end{equation}
In the generic regime
\begin{equation}\label{eq:generic-regime}
c_{12}\neq 0,\qquad c_{13}\neq 0,\qquad c_{23}\neq 0,
\end{equation}
one reads off from \eqref{eq:X1}--\eqref{eq:X3} and \eqref{eq:operator} that $\widetilde{\mathcal F}$ is tangent along $\pi^{-1}(p)$ to all three angular directions $\partial_{\theta_1},\partial_{\theta_2},\partial_{\theta_3}$, hence $k=3$ and
\begin{equation}\label{eq:morita-generic-explicit}
\Mon^{\mathrm{tw}}(\mathcal F;D)\big|_{\pi^{-1}(p)}\ \simeq_{\mathrm{Morita}}\ \mathbb{Z}^3\times\mathbb{Z}^3.
\end{equation}
Finally, assume we are given a Poisson-flat connection on a rank-$e$ bundle over $U$, hence a leafwise monodromy representation $\rho:\Mon(\mathcal F)|_U\to \mathrm{GL}_e(\mathbb{C})$, and assume that Euler--Poisson (PPD) normalization in an effectively non-resonant regime puts the logarithmic part of the connection in the normal form
\begin{equation}\label{eq:normal-form}
\Theta_{\mathrm{nf}}=\sum_{i=1}^3 \widetilde A_i\,X_i,
\qquad
[\widetilde A_i,\widetilde A_j]=0,
\qquad
\widetilde A_i\in\mathrm{Mat}_{e\times e}(\mathbb{C})\ \text{constant}.
\end{equation}
For $i=1,2,3$, let $t_i$ be the loop in the corner torus \eqref{eq:corner-fiber-explicit} obtained by fixing all variables except $\theta_i$ and letting $\theta_i$ run once from $0$ to $2\pi$; then $\int_{t_i} dz_i/z_i=2\pi i$, so the parallel-transport convention for $\nabla=d+\omega$ gives
\begin{equation}\label{eq:holonomy-meridian}
\rho^{\mathrm{tw}}(t_i)=\exp(-2\pi i\,\widetilde A_i),\qquad i=1,2,3,
\end{equation}
and, more generally, the commuting residue matrices define a character of the canonical meridional group $\mathbb{Z}^3$ by
\begin{equation}\label{eq:meridional-character}
(m_1,m_2,m_3)\longmapsto
\exp\Bigl(-2\pi i\sum_{i=1}^3 m_i\,\widetilde A_i\Bigr).
\end{equation}
Equivalently, the normal form provides the meridional boundary character \eqref{eq:meridional-character}. Together with the interior leafwise monodromy representation $\rho$, it gives a representation of the boundary-twisted groupoid whenever the boundary leafwise transport is compatible with the simultaneous centralizer of the residue tuple. Its restriction to the canonical meridional subgroupoid \eqref{eq:corner-bundle-groupoid-explicit} is exactly \eqref{eq:meridional-character}.

\begin{figure}[t]
\centering
\begin{tikzcd}[row sep=3.2em, column sep=4.0em]
\Mon(\widetilde{\mathcal F})|_{U\cap C}
\arrow[r, "{\gamma\mapsto(\gamma,0)}"]
\arrow[d, hook]
&
\mathcal C^{\mathrm{tw}}
\arrow[d]
\\
\Mon(\mathcal F)|_U
\arrow[r]
&
\Mon^{\mathrm{tw}}(\mathcal F;D)
\end{tikzcd}
\caption{Tangential completion as a pushout: the collar is enlarged by meridional labels only on the oriented boundary, and the overlap with the interior carries zero meridional label.}
\label{fig:pushout-diagram}
\end{figure}

\section*{Acknowledgements}
MC is partially supported by the Universit\`a degli Studi di Bari and by PRIN 2022MWPMAB, ``Interactions between Geometric Structures and Function Theories.'' He is a member of INdAM-GNSAGA.

\bibliographystyle{alpha}

\end{document}